%% file: main_LCprep.tex
\documentclass[12pt]{amsart}
\usepackage{amssymb, amsfonts,amscd,graphicx}
\usepackage[all]{xy}
\usepackage{enumerate}
\input{macros}

\author{Raf Cluckers}
\address{Universit\'e Lille 1, Laboratoire Painlev\'e, CNRS - UMR 8524, Cit\'e Scientifique, 59655
Villeneuve d'Ascq Cedex, France, and,
KU Leuven, Department of Mathematics,
Celestijnenlaan 200B, B-3001 Leu\-ven, Bel\-gium}
\email{Raf.Cluckers@math.univ-lille1.fr}
\urladdr{http://math.univ-lille1.fr/$\sim$cluckers}

\author[Miller]{Daniel~J.~Miller}
\address{Emporia State University, Department of Mathematics, Computer Science and Economics, 1200 Commercial Street, Campus Box 4027, Emporia, KS 66801, U.S.A.}
\email{dmille10@emporia.edu}

\title{Lebesgue classes and preparation of real constructible functions}


\begin{document}

\begin{abstract}
We call a function \emph{constructible} if it has a globally subanalytic domain and can be expressed as a sum of products of globally subanalytic functions and logarithms of positively-valued globally subanalytic functions.  For any $q > 0$ and constructible functions $f$ and $\mu$ on $E\times\RR^n$, we prove a theorem describing the structure of the set
\[
\{(x,p)\in E\times(0,\infty] : f(x,\cdot) \in L^p(|\mu|_{x}^{q})\},
\]
where $|\mu|_{x}^{q}$ is the positive measure on $\RR^n$ whose Radon-Nikodym derivative with respect to the Lebesgue measure is $|\mu(x,\cdot)|^q : y\mapsto |\mu(x,y)|^q$.  We also prove a closely related preparation theorem for $f$ and $\mu$.  These results relate analysis (the study of $L^p$-spaces) with geometry (the study of zero loci). \vspace*{-5pt}\\

\end{abstract}

\maketitle

\input{intro}

\input{mainResults}

\input{subPrep}

\input{cor}

\input{vanish}

\input{subRect}

\input{constrRect}

\input{mainThmProofs}

\input{countEx}

\input{acknowledgment}

\bibliographystyle{amsplain}
\bibliography{bibliotex}
\end{document}

%% file: macros.tex
\newtheoremstyle{slthm}
  {9pt}
  {5pt}
  {\slshape}
  {}
  {\bfseries}
  {.}
  {.5em}
  {\thmname{#1} \thmnumber{#2}{\rm \thmnote{ (#3)}}}

\newtheoremstyle{Slthm}
  {9pt}
  {5pt}
  {\slshape}
  {}
  {\bfseries}
  {.}
  {0.5em}
  {\thmname{#1}{\rm \thmnote{ (#3)}}}

\newtheoremstyle{prcl}
  {9pt}
  {5pt}
  {\slshape}
  {}
  {\bfseries}
  {.}
  {.5em}
  {\thmname{#3} \thmnumber{ #2}}

\newtheoremstyle{slthmprime}
  {9pt}
  {5pt}
  {\slshape}
  {}
  {\bfseries}
  {.}
  {.5em}
  {\thmname{#1} \thmnumber{#2}${}^\prime${\rm \thmnote{ (#3)}}}

\theoremstyle{slthm}
\newtheorem{theorem}{Theorem}[section]
\newtheorem{lemma}[theorem]{Lemma}

\newtheorem{proposition}[theorem]{Proposition}
\newtheorem{corollary}[theorem]{Corollary}

\newtheorem{assertion}[theorem]{Assertion}

\theoremstyle{Slthm}

\theoremstyle{definition}
\newtheorem{definition}[theorem]{Definition}
\newtheorem{definitions}[theorem]{Definitions}

\theoremstyle{remark}

\newtheorem{notation}[theorem]{Notation}

\newtheorem{remark}[theorem]{Remark}
\newtheorem{remarks}[theorem]{Remarks}

\newenvironment{caselist}{
	\begin{list}{Case \arabic{enumi}:~}{\usecounter{enumi} \setlength{\itemsep}{3pt}
                \setlength{\leftmargin}{30pt} \setlength{\labelwidth}{10pt}}
}{
	\end{list}
}

\numberwithin{equation}{section}

\allowdisplaybreaks[2]

\newcommand{\A}{\mathcal{A}}

\newcommand{\C}{\mathcal{C}}

\newcommand{\F}{\mathcal{F}}
\newcommand{\G}{\mathcal{G}}
\newcommand{\I}{\mathcal{I}}
\newcommand{\J}{\mathcal{J}}

\renewcommand{\O}{\mathcal{O}}   %
\renewcommand{\P}{\mathcal{P}} %

\renewcommand{\S}{\mathcal{S}}   %

\newcommand{\U}{\mathcal{U}}

\newcommand{\NN}{\mathbb{N}}
\newcommand{\ZZ}{\mathbb{Z}}
\newcommand{\QQ}{\mathbb{Q}}
\newcommand{\RR}{\mathbb{R}}

\newcommand{\VV}{\mathbb{V}}

\renewcommand{\bar}[1]{\overline{#1}} %
\newcommand{\tld}[1]{\widetilde{#1}}

\DeclareMathOperator{\Span}{span}

\newcommand{\Restr}[1]{\big|_{#1}}

\newcommand{\PD}[3]{\frac{\partial^{#1}#2}{\partial {#3}^{#1}}}

\renewcommand{\implies}{\rightarrow}
\DeclareMathOperator{\supp}{supp}

\DeclareMathOperator{\LC}{LC}
\DeclareMathOperator{\Int}{Int}

\DeclareMathOperator{\loc}{loc}

\DeclareMathOperator{\CR}{cr}
\DeclareMathOperator{\NC}{nc}

%% file: intro.tex
\section*{Introduction}\label{s:intro}

The Lebesgue spaces, $L^p(\mu)$ for $p\in(0,\infty]$, are ubiquitous in many areas of mathematical analysis and its applications.  Much of the research about the Lebesgue spaces has been conducted in a very general measure-theoretic framework, with the focus being on discovering a host of relationships between the various $L^p$ spaces.  A number of the classical theorems are inequalities that explain how various function operations behave with respect to the Lebesgue spaces.  For example, for addition there is Minkowski's inequality; for multiplication there is H\"{o}lder's inequality; for convolutions there is Young's convolution inequality; for Fourier transforms of periodic functions there is the Hausdorff-Young inequality.  Other classical theorems explain the structure of linear maps between the various $L^p$ spaces, such as the duality of Lebesgue spaces with conjugate exponents and the Riesz-Thorin interpolation theorem.

This paper explores theorems about the Lebesgue spaces of a rather different sort.  We use geometric techniques to study the structure of the Lebesgue classes of parameterized families of functions, along with a related preparation theorem.  The starting point of our investigation is the observation that, although much of the utility of the Lebesgue spaces --- and more generally, of the theory of integration as a whole --- stems from the generality of the measure-theoretic framework in which it has been developed, it is many times applied to study integrals of very special functions that arise naturally in real analytic geometry.  And, if we focus our attention on studying the $L^p$ properties of these very special functions, we should be able to obtain rather strong theorems that cannot be proven, or even reasonably formulated, in a very general measure-theoretic framework.  This is because by focusing on special functions, we can supplement the very general tools from mathematical analysis with much more specialized tools from real analytic geometry and o-minimal structures. Similar approaches have been followed in the context of $p$-adic and motivic integration; see e.g.~\cite{CLoes}.

The o-minimal framework is still a bit too general for our purposes, and we choose to focus on the constructible functions, by which we mean the real-valued functions that have globally subanalytic domains and that can be expressed as sums of products of globally subanalytic functions and logarithms of positively-valued globally subanalytic functions.  The study of constructible functions largely originated in the work of Lion and Rolin, \cite{LR98}, where these functions naturally arose in their study of integration of globally subanalytic functions. (In the context of $p$-adic integration, analogues of constructible functions arose from the work by J.~Denef \cite{Denef}.)
The integration theory of globally subanalytic and constructible functions was then further developed by Comte, Lion and Rolin in \cite{CLR2000} and also by the authors in \cite{CluckersMiller1} and \cite{CluckersMiller2}.  Much of the utility of the constructible functions stems from the fact that they are stable under integration --- from which it follows that they are the smallest class of functions that is stable under integration and contains the subanalytic functions --- and that they have very simple asymptotic behavior (see Theorem 1.3 and Proposition 1.5 in \cite{CluckersMiller1}). In fact, these results have typically lagged behind the motivic and $p$-adic developments. In this paper, the real situation takes the lead over the $p$-adic and motivic results.

We obtain two main theorems about the constructible functions; see Theorems \ref{thm:LC} and \ref{thm:constrPrepSimple}. The first theorem considers a constant $q > 0$ and constructible functions $f$ and $\mu$ on $E\times\RR^n$, and it describes the structure of the set
\begin{equation}\label{eq:LCintro}
\LC(f,|\mu|^q,E) := \{(x,p)\in E\times(0,\infty] : f(x,\cdot) \in L^p(|\mu|_{x}^{q})\},
\end{equation}
where $|\mu|_{x}^{q}$ is the positive measure on $\RR^n$ whose Radon-Nikodym derivative with respect to the Lebesgue measure is $|\mu(x,\cdot)|^q : y\mapsto |\mu(x,y)|^q$.  The theorem and its corollaries show that the set of all fibers of $\LC(f,|\mu|^q,E)$ over $E$ is a finite set of open subintervals of $(0,\infty]$, and that the set of all fibers of $\LC(f,|\mu|^q,E)$ over $(0,\infty]$ is a finite set of subsets of $E$, each of which is the zero locus of a constructible function on $E$.  This theorem therefore relates analysis with geometry, in the sense that Lebesgue classes are an object of study in analysis, while zero loci of functions are widely studied in analytic geometry. A similar link between geometry and analysis (but with $\mu=1$  and with focus on $L^1$-integrability) is obtained in $p$-adic and motivic contexts in \cite{CGH}.

The second theorem is a closely related preparation result that expresses $f$ and $\mu$ as finite sums of terms of a very simple form that naturally reflect the structure of $\LC(f,|\mu|^q,E)$.  This theorem can be most easily appreciated through the  historical context in which it was developed, starting with the following simple preparation result for constructible functions, which is a rather direct consequence of Lion and Rolin's preparation theorem for globally subanalytic functions:
\begin{equation}\label{eq:formalConstrPrep}
\left\{\text{\parbox{5.8in}{
Let $f:E\times\RR^n\to\RR$ be constructible, with $E\subset\RR^m$, and write $(x,y) = (x_1,\ldots,x_m,y_1,\ldots,y_n)$ for the standard coordinates on $E\times\RR^n$.  Then $f$ can be piecewise written on subanalytic sets as finite sums $\textstyle \sum_{k\in K} T_k(x,y)$, where up to performing translations in $y$ by globally subanalytic functions of a triangular form, each term is of the form $\textstyle T_k(x,y) = g_k(x)\left(\prod_{j=1}^{n}|y_{j}|^{r_{k,j}} (\log|y_j|)^{s_{k,j}}\right) u_k(x,y)$ for some constructible function $g_k$, rational numbers $r_{k,j}$, natural numbers $s_{k,j}$, and globally subanalytic unit $u_k$ which is of the special form as given by the globally subanalytic preparation theorem.}}\right.
\end{equation}
Lion and Rolin \cite{LR97} used \eqref{eq:formalConstrPrep} when proving that any parameterized integral of a constructible function is piecewise given by constructible functions, but on pieces that need not be globally subanalytic sets.  Comte, Lion and Rolin \cite{CLR2000} also used \eqref{eq:formalConstrPrep} when proving that any parameterized integral of a globally subanalytic function is a constructible function.  The authors then subsumed both of these results in \cite{CluckersMiller1} by showing that $F(x) = \int_{\RR^n}f(x,y)dy$ is a constructible function on $E$ if $f:E\times\RR^n\to\RR$ is a constructible function such that $f(x,\cdot)\in L^1(\RR^n)$ for all $x\in E$.  The key to doing this was to improve \eqref{eq:formalConstrPrep} by showing that in the special case of $n=1$, if $f(x,\cdot)\in L^1(\RR)$ for every $x\in E$, then the sums can be constructed in such a way so that each term $T_k(x,y)$ is also integrable in $y$ for every $x\in E$.  This alleviated various analytic considerations employed in \cite{LR97} and \cite{CLR2000} to get around the awkward fact that \eqref{eq:formalConstrPrep} allows the possibility of expressing integrable functions as sums of nonintegrable functions.  In \cite{CluckersMiller2} the authors again improved upon \eqref{eq:formalConstrPrep} in the special case of $n=1$ by dropping the assumption that $f(x,y)$ be integrable in $y$ for every $x\in E$, and then showing that the set $\Int(f,E) := \{x\in E : f(x,\cdot)\in L^1(\RR)\}$ is the zero locus of a constructible function on $E$, and that the sums in \eqref{eq:formalConstrPrep} can be constructed so that each term $T_k(x,y)$ is integrable in $y$ for every $x\in E$, provided that we only require the equation $f(x,y) = \sum_k T_k(x,y)$ to hold for those values of $(x,y)$ with $x\in \Int(f,E)$.

The preparation theorem of this paper strengthens this line of results even further by considering an arbitrary positive integer $n$, not just $n=1$, and by considering all $L^p$ classes simultaneously, not just $L^1$.  In order to convey the main idea of the theorem without getting bogged down in technicalities, let us use the Lebesgue measure on $\RR^n$ (thus $\mu=1$, where $\mu$ is the function from \eqref{eq:LCintro}), and let us also only consider the $L^p$ classes for finite values of $p$.  Under these simplifying assumptions, the preparation theorems states that the sums $\sum_{k\in K} T_k(x,y)$ in \eqref{eq:formalConstrPrep} can be constructed in such a way so that there is a partition $\{K_i\}_i$ of the finite index set $K$ such that for each $x\in E$ and $p\in(0,\infty)$ with $f(x,\cdot)\in L^p(\RR^n)$, and for each $i$, either $T_k(x,\cdot)$ is in $L^p$ for all $k\in K_i$, or else $\sum_{k\in K_i} T_k(x,y) = 0$ for all $y$.  So, for instance, if for some fixed value of $p$ the function $f(x,\cdot)$ happened to be in $L^p(\RR^n)$ for every $x\in E$, then the sums in \eqref{eq:formalConstrPrep} can be constructed so that each term $T_k(x,\cdot)$ is in $L^p$ for every $x\in E$, for we may simply omit the remaining terms in the sum because they collectively sum to zero.

Part of our interest in developing a good integration theory for constructible functions comes from a desire to study various integral transforms in the constructible setting.  And, to summarize, we now have three main tools at our disposal to conduct such studies: the constructible functions are stable under integration, they have simple asymptotic behavior, and they have a multivariate preparation theorem with good analytic properties.   We apply these three tools to the field of harmonic analysis in \cite{CluckersMiller4} by proving a theorem that bounds the decay rates of parameterized families of oscillatory integrals.  This is an adaptation of a classical theorem found in Stein \cite[Chapter VIII, Section 3.2]{Stein} but with different assumptions.  The classical theorem bounds a single oscillatory integral with an amplitude function that is smooth and compactly supported and a phase function that is smooth and of finite type.  In contrast, we give a uniform bound on a parameterized family of oscillatory integrals with an amplitude function that is constructible and integrable and a phase function that is globally subanalytic and satisfies a certain ``hyperplane condition'' (which closely relates to the notion of ``finite type'' in our setting).  Thus by restricting our attention to the special classes of constructible and globally subanalytic functions, we obtain a much more global, parameterized version of the classical theorem with significantly weaker analytic assumptions.  This application of our preparation theorem was, in fact, the initial stimulus for our work in this paper.

%% file: mainResults.tex
\section{The Main Results}\label{s:mainResults}

This section formulates our main theorem on the structure of diagrams of Lebesgue classes and also a simple version of the related preparation theorem; see Theorems \ref{thm:LC} and \ref{thm:constrPrepSimple}. It also gives two key supporting theorems used to prove these results; see Theorems \ref{thm:vanish} and \ref{thm:subRect}. The full version of the preparation theorem can be found in Section \ref{s:mainThmProofs} as Theorem \ref{thm:constrPrep}. We begin by fixing some notation to be used throughout the paper.

\begin{notation}
Denote the set of natural numbers by $\NN=\{0,1,2,3,\ldots\}$.  Denote the subset and proper subset relations by $\subset$ and $\subsetneq$, respectively.  Write $x = (x_1,\ldots,x_m)$ and $y = (y_1,\ldots,y_n)$ for the standard coordinates on $\RR^m$ and $\RR^n$, respectively.  If $f = (f_1,\ldots,f_n):D\to\RR^n$ is a differentiable map with $D\subset\RR^{m+n}$, write
\[
\PD{}{f}{y}(x,y) = \left(\PD{}{f_i}{y_j}(x,y)\right)_{(i,j)\in\{1,\ldots,n\}^2}
\]
for its Jacobian matrix in $y$.  Define the coordinate projection $\Pi_m:\RR^{m+n}\to\RR^n$ by
\[
\Pi_m(x,y) = x.
\]
For any $D\subset\RR^{m+n}$ and $x\in\RR^m$, define the \emph{fiber of $D$ over $x$} by
\[
D_x = \{y\in\RR^n : (x,y)\in D\}.
\]
For any $d\in\{0,\ldots,n\}$ and $\Box\in\{<,\leq,>,\geq\}$, define $y_{\Box d} = (y_i)_{i\Box d}$.  For example, $y_{\leq d} = (y_1,\ldots,y_d)$, and in accordance with our above notation for coordinate projections, the maps $\Pi_d:\RR^n\to\RR^d$ and $\Pi_{m+d}:\RR^{m+n}\to\RR^{m+d}$ are given by $\Pi_d(y) = y_{\leq d}$ and $\Pi_{m+d}(x,y) = (x,y_{\leq d})$.  More generally, if $\lambda:\{1,\ldots,d\}\to\{1,\ldots,n\}$ is an increasing map, define $\Pi_{m,\lambda}:\RR^{m+n}\to\RR^{m+d}$ by
\[
\Pi_{m,\lambda}(x,y) = (x,y_\lambda),
\]
where $y_\lambda = (y_{\lambda(1)},\ldots,y_{\lambda(d)})$.
\end{notation}

For any set $D\subset\RR^n$, call a function $f:D\to\RR^m$ {\bf\emph{ analytic}} if it extends to an analytic function on a neighborhood of $D$ in $\RR^n$.  A {\bf\emph{restricted analytic function}} is a function $f:\RR^n\to\RR$ such that the restriction of $f$ to $[-1,1]^n$ is analytic and $f(x) = 0$ on $\RR^n\setminus[-1,1]^n$.  We shall henceforth call a set or function {\bf\emph{subanalytic}} if, and only if, it is definable (in the sense of first-order logic) in the expansion of the real field by all restricted analytic functions.  Thus in this paper, the word ``subanalytic'' is an abbreviation for the phrase ``globally subanalytic'', and in this meaning, the natural logarithm $\log:(0,\infty)\to\RR$ is not subanalytic.  For any subanalytic set $D$, let $\C(D)$ denote the $\RR$-algebra of functions on $D$ generated by the functions of the form $x\mapsto f(x)$ and $x\mapsto \log g(x)$, where $f:D\to\RR$ and $g:D\to(0,\infty)$ are subanalytic.  A function that is a member of $\C(D)$ for some subanalytic set $D$ is called a {\bf\emph{constructible function}}.\footnote{We use the phrase \emph{``real'' constructible functions} in our title to distinguish them from an analogous notion of constructible functions in the $p$-adic setting, from which we borrow the terminology.}

Consider a Lebesgue measurable set $D\subset\RR^{m+n}$ and Lebesgue measurable functions $f:D\to\RR$ and $\nu:D\to[0,\infty)$, and put $E = \Pi_m(D)$.  Define the {\bf \emph{diagram of Lebesgue classes of $f$ over $E$ with respect to $\nu$}} to be the set
\[
\LC(f,\nu,E) = \{(x,p)\in E\times(0,\infty] : f(x,\cdot)\in L^p(\nu_x)\},
\]
where $\nu_x$ is the positive measure on $D_x$ defined by setting
\begin{equation}\label{eq:nu}
\nu_x(Y) = \int_Y \nu(x,y) dy
\end{equation}
for each Lebesgue measurable set $Y\subset D_x$, where the integration in \eqref{eq:nu} is with respect to the Lebesgue measure on $\RR^n$.  Thus for each $x\in E$, when $0 < p < \infty$, the function $f(x,\cdot)$ is in $L^p(\nu_x)$ if and only if
\[
\int_{D_x} |f(x,y)|^p \nu(x,y) dy < \infty,
\]
and the function $f(x,\cdot)$ is in $L^\infty(\nu_x)$ if and only if there exist a constant $M > 0$ and a Lebesgue measurable set $Y\subset D_x$ such that $\nu_x(Y) = 0$ and $|f(x,y)| \leq M$ for all $y\in D_x\setminus Y$.

The fibers of $\LC(f,\nu,E)$ over $E$ and over $(0,\infty]$ are both of interest, so we give them special names.  For each $x\in E$, define the {\bf\emph{set of Lebesgue classes of $f$ at $x$ with respect to $\nu$}} to be the set
\[
\LC(f,\nu,x) = \{p\in(0,\infty] : f(x,\cdot)\in L^p(\nu_x)\}.
\]
For each $p\in(0,\infty]$, define the {\bf\emph{$L^p$-locus of $f$ in $E$ with respect to $\nu$}} to be the set
\[
\Int^p(f,\nu,E) = \{x\in E : f(x,\cdot)\in L^p(\nu_x)\}.
\]
When $\nu=1$ (which is the case of most interest because it means we are simply using the $n$-dimensional Lebesgue measure on $D_x$), it is convenient to simply write $\LC(f,E)$, $\LC(f,x)$ and $\Int^p(f,E)$ and to drop the phrase ``with respect to $\nu$'' in the names of theses sets.  Also when $\nu=1$, we shall write $L^p(D_x)$ rather than $L^p(\nu_x)$.  The set $\Int^1(f,E)$ was studied by the authors in \cite{CluckersMiller2} (focusing on the case of $n=1$), where it was denoted by $\Int(f,E)$ and called the ``locus of integrability of $f$ in $E$.''

We order the set $[0,\infty]$ in the natural way, and we topologize $(0,\infty]$ by letting
\[
\{(a,b) : 0\leq a < b < \infty\}\cup\{\{\infty\}\}
\]
be a base for its topology.  A convex subset of $(0,\infty]$ is called a {\bf\emph{subinterval of $(0,\infty]$}}.  The {\bf\emph{endpoints}} of a subinterval of $(0,\infty]$ are its supremum and infimum in $[0,\infty]$.  Note that the empty set is a subinterval of $(0,\infty]$, and that $\sup\emptyset = 0$ and $\inf\emptyset = \infty$.

It is elementary to see that $\LC(f,\nu,x)$ is a subinterval of $(0,\infty]$ for each $x\in E$.  Much more can be said when $f$ and $\nu$ are assumed to be constructible functions or their powers.

\begin{theorem}[The Structure of Diagrams of Lebesgue Classes]\label{thm:LC}
Let $q > 0$ and $f,\mu\in\C(D)$ for some subanalytic set $D\subset\RR^{m+n}$, and put $E = \Pi_m(D)$ and $\I = \{\LC(f,|\mu|^q,x) : x\in E\}$.  Then $\I$ is a finite set of open subintervals of $(0,\infty]$ with endpoints in $\left(\Span_{\QQ}\{1,q\}\cap[0,\infty)\right)\cup\{\infty\}$, and for each $I\in\I$ there exists $g_I\in\C(E)$ such that
\begin{equation}\label{eq:LC}
\{x\in E : I\subset\LC(f,|\mu|^q,x)\} = \{x\in E : g_I(x) = 0\}.
\end{equation}
Moreover, if $f$ and $\mu$ are subanalytic, then each of the functions $g_I$ can be taken to be subanalytic.
\end{theorem}

Theorem \ref{thm:LC} has been formulated in such a way so as to make it adaptable to a variety of situations.  Section \ref{s:cor} contains an extensive list of corollaries that further explain how the theorem elucidates the structure of $\LC(f,|\mu|^q,E)$, and how it can be easily adapted to give analogous theorems about local $L^p$ spaces, complex measures, and measures defined from differential forms on subanalytic sets, all within the context of constructible functions.

The proof of Theorem \ref{thm:LC} is intimately linked to the proof of a preparation theorem for constructible functions that is stated in full strength in Section \ref{s:mainThmProofs}, where it is proved.  Here we state only a simple version of the preparation theorem that is sufficient for our application to oscillatory integrals in \cite{CluckersMiller4}.  But first, we need one more definition: a {\bf\emph{cell over $\RR^m$}} is a subanalytic set $A\subset\RR^{m+n}$ such that for each $i\in\{1,\ldots,n\}$, the set $\Pi_{m+i}(A)$ is either the graph of an analytic subanalytic function on $\Pi_{m+i-1}(A)$, or
\begin{equation}\label{eq:cell}
\Pi_{m+i}(A) = \{(x,y_{\leq i}) : (x,y_{<i})\in\Pi_{m+i-1}(A), a_i(x,y_{<i})\,\, \Box_1 \,\, y_i \,\,\Box_2 \,\, b_i(x,y_{<i})\}
\end{equation}
for some analytic subanalytic functions $a_i,b_i:\Pi_{m+i-1}(A)\to\RR$ for which $a_i(x,y_{<i}) < b_i(x,y_{<i})$ on $\Pi_{m+i-1}(A)$, where $\Box_1$ and $\Box_2$ denote either $<$ or no condition.

\begin{theorem}[Preparation of Constructible Functions - Simple Version]\label{thm:constrPrepSimple}
Let $p\in(0,\infty)$ and $f\in\C(D)$ for some subanalytic set $D\subset\RR^{m+n}$, and assume that $\Int^p(f,\Pi_m(D)) = \Pi_m(D)$.  Then there exists a finite partition $\A$ of $D$ into cells over $\RR^m$ such that for each $A\in\A$ whose fibers over $\Pi_m(A)$ are open in $\RR^n$, we may write $f$ as a finite sum
\[
f(x,y) = \sum_{k} T_k(x,y)
\]
on $A$, with $\Int^p(T_k,\Pi_m(A)) = \Pi_m(A)$ for each $k$, as follows: there exists a bounded function $\varphi:A\to(0,\infty)^M$ of the form
\[
\varphi(x,y)
=
\left(
c_i(x)\prod_{j=1}^{n}|y_j-\theta_j(x,y_{<j})|^{\gamma_{i,j}}
\right)_{\!\!i\in\{1,\ldots,M\}},
\]
and for each $k$,
\begin{equation}\label{eq:constrPrepSimple}
T_k(x,y)
=
g_k(x)
\left(\prod_{i=1}^{n}|y_{i} - \theta_i(x,y_{<i})|^{r_{k,i}} \left(\log|y_i-\theta_i(x,y_{<i})|\right)^{s_{k,i}}
\right)
U_k\circ\varphi(x,y),
\end{equation}
where the $g_k:\Pi_m(A)\to\RR$ are constructible, the $c_i:\Pi_m(A)\to(0,\infty)$ and $\theta_i:\Pi_{m+i-1}(A)\to\RR$ are analytic subanalytic functions, the graph of each $\theta_i$ is disjoint from $\Pi_{m+i}(A)$, the $\gamma_{i,j}$ and $r_{k,i}$ are rational numbers, the $s_{k,i}$ are natural numbers, and the $U_k$ are positively-valued analytic functions on the closure of the range of $\varphi$.

In addition, the fact that $\Int^p(T_k,\Pi_m(A)) = \Pi_m(A)$ only depends on the values of the $r_{k,i}$, and not the values of $s_{k,i}$, in the following sense: we have $\Int(T'_k,\Pi_m(A)) = \Pi_m(A)$ for any function $T'_k$ on $A$ of the form
\[
T'_k(x,y) = \prod_{i=1}^{n}|y_{i} - \theta_i(x,y_{<i})|^{r_{k,i}} \left(\log|y_i-\theta_i(x,y_{<i})|\right)^{s'_{k,i}},
\]
where the $r_{k,j}$ are as in \eqref{eq:constrPrepSimple} and the $s'_{k,i}$ are arbitrary natural numbers.
\end{theorem}

The key aspect of Theorem \ref{thm:constrPrepSimple} that is of interest, and what makes its proof nontrivial, is that the piecewise sum representation of $f$ can be constructed so that each of its terms $T_k(x,\cdot)$ are in the same $L^p$ class as $f(x,\cdot)$; namely, $\Int^p(T_k,\Pi_m(A)) = \Pi_m(A)$ for each $A$ and $T_k$, provided that $\Int^p(f,\Pi_m(D)) = \Pi_m(D)$.  There is an analog of Theorem \ref{thm:constrPrepSimple} for $p=\infty$, but then one must replace \eqref{eq:constrPrepSimple} with the more complicated form
\begin{equation}\label{eq:constrPrepSimpleBdd}
T_k(x,y)
=
g_k(x)
\left(
\prod_{i=1}^{n}|y_{i} - \theta_i(x,y_{<i})|^{r_{k,i}}
\left(
\log \prod_{j=1}^{n} |y_j-\theta_j(x,y_{<j})|^{\beta_{i,j}}
\right)^{s_{k,i}}
\right)
U_k\circ\varphi(x,y),
\end{equation}
where the $\beta_{i,j}$ are rational numbers and everything else is as before, and where the fact that $\Int^\infty(T_k,\Pi_m(A)) = \Pi_m(A)$ now depends on all the values of the $r_{k,i}$, $s_{k,i}$ and $\beta_{i,j}$, not just the values of the $r_{k,i}$ alone.

In the course of proving our main results, Theorems \ref{thm:LC} and \ref{thm:constrPrepSimple}, we shall also prove a theorem on the fiberwise vanishing of constructible functions and a theorem on parameterized rectilinearization of subanalytic functions, given below.

\begin{theorem}[Fiberwise Vanishing of Constructible Functions]\label{thm:vanish}
If $f\in\C(D)$ for a subanalytic set $D\subset\RR^{m+n}$ and $E = \Pi_m(D)$, then there exists $g\in\C(E)$ such that
\[
\{x\in E : \text{$f(x,y) = 0$ for all $y\in D_x$}\}
=
\{x\in E : g(x) = 0\}.
\]
\end{theorem}

The parameterized rectilinearization theorem requires some additional terminology to state.  For any sets $A\subset\RR^{m+n}$ and $B\subset\RR^{m+d}$, we call a map $f = (f_1,\ldots,f_{m+n}):B\to A$ an {\bf\emph{analytic isomorphism over $\RR^m$}} if $f$ is a bijection, $f$ and $f^{-1}$ are both analytic, and $f_1(x,z) = x_1$, \ldots, $f_m(x,z) = x_m$, where $z = (z_1,\ldots,z_d)$.

For $l\in\{0,\ldots,d\}$, we say that a set $B\subset\RR^{m+d}$ is {\bf\emph{$l$-rectilinear over $\RR^m$}} if $B$ is a cell over $\RR^m$ such that for each $x\in\Pi_m(B)$, the fiber $B_x$ is an open subset of $(0,1)^d$ of the form
\[
B_x = \Pi_l(B_x)\times(0,1)^{d-l},
\]
where the closure of $\Pi_l(B_x)$ is a compact subset of $(0,1]^l$.  When $B\subset\RR^{m+d}$ is $l$-rectilinear over $\RR^m$, we call a function $u$ on $B$ an {\bf\emph{$l$-rectilinear unit}} if it may written in the form $u = U\circ\psi$, where $\psi:B\to(0,\infty)^{N+d-l}$ is a bounded function of the form
\begin{equation}\label{eq:rectUnit}
\psi(x,z) = \left(c_1(x)\prod_{j=1}^{l}z_{j}^{\gamma_{1,j}}, \ldots, c_N(x)\prod_{j=1}^{l}z_{j}^{\gamma_{N,j}},z_{l+1},\ldots,z_d\right)
\end{equation}
for some positively-valued analytic subanalytic functions $c_i$ and rational numbers $\gamma_{i,j}$, and where $U$ is a positively-valued analytic function on the closure of the range of $\psi$.

\begin{theorem}[Parameterized Rectilinearization of Subanalytic Functions]
\label{thm:subRect}
Let $\F$ be a finite set of subanalytic functions on a subanalytic set $D\subset\RR^{m+n}$.  Then there exists a finite partition $\A$ of $D$ into subanalytic sets such that for each $A\in\A$ there exist $d\in\{0,\ldots,n\}$, $l\in\{0,\ldots,d\}$ and a subanalytic map $F:B\to A$ such that $F$ is an analytic isomorphism over $\RR^m$, the set $B\subset\RR^{m+d}$ is $l$-rectilinear over $\RR^m$, and each function $g$ in the set $\G$ defined by
\[
\G =
\begin{cases}
\{f\circ F\}_{f\in \F},
    & \text{if $d < n$,} \\
\{f\circ F\}_{f\in \F}\cup\{\det\PD{}{F}{y}\},
    & \text{if $d = n$,}
\end{cases}
\]
may be written in the form
\begin{equation}\label{eq:gRect}
g(x,z) = h(x)\left(\prod_{j=1}^{d}z_{j}^{r_j}\right) u(x,z)
\end{equation}
on $B$ for some analytic subanalytic function $h$, rational numbers $r_j$, and $l$-rectilinear unit $u$.
\end{theorem}

Note that if one desires, one can take the $\gamma_{i,j}$ in \eqref{eq:rectUnit} and the $r_j$ in \eqref{eq:gRect} to all be integers.  To do this, simply pull back each map $F$ in Theorem \ref{thm:subRect} by a map $(x,z)\mapsto (x,z_{1}^{k_1},\ldots,z_{d}^{k_d})$ for a suitable choice of positive integers $k_1,\ldots,k_d$.

We now conclude this section with an outline of the rest of the paper.  Section \ref{s:subPrep} formulates a version of the subanalytic preparation theorem of Lion and Rolin \cite{LR97}, which is one of our main tools.  Section \ref{s:cor} gives an extensive list of corollaries of Theorem \ref{thm:LC}.  Section \ref{s:vanish} proves Theorem \ref{thm:vanish}.  Section \ref{s:subRect} states and proves Proposition \ref{prop:subRect}, which is a slightly more detailed version of Theorem \ref{thm:subRect}, and this is used to prove Theorem \ref{thm:LC} in the special case when $f$ and $\mu$ are both subanalytic.  Section \ref{s:constrRect} uses Proposition \ref{prop:subRect} to prove a preparation result for constructible functions in transformed coordinates on rectilinear sets.  Section \ref{s:mainThmProofs} uses Theorem \ref{thm:vanish} and the preparation result on rectilinear sets to prove Theorem \ref{thm:LC} in the general case when $f$ and $\mu$ are both constructible; and by pushing forward this preparation result to the original coordinates, it also proves a preparation theorem for constructible functions, of which Theorem \ref{thm:constrPrepSimple} and its analog for $p=\infty$ described in \eqref{eq:constrPrepSimpleBdd} are special cases.  The paper concludes in Section \ref{s:countEx}, which gives an example that shows the necessity of allowing terms of form \eqref{eq:constrPrepSimpleBdd}, rather than \eqref{eq:constrPrepSimple}, in the analog of Theorem \ref{thm:constrPrepSimple} for $p=\infty$.

%% file: subPrep.tex
\section{The Subanalytic Preparation Theorem}
\label{s:subPrep}

This section formulates a version of the subanalytic preparation theorem of Lion and Rolin \cite{LR97}.  We begin with some multi-index notation.

\begin{notation}\label{notation:multindex}
For any tuples $y = (y_1,\ldots,y_n)$ and $\alpha = (\alpha_1,\ldots,\alpha_n)$ in $\RR^n$, define
\begin{eqnarray*}
|y|
    & = &
    (|y_1|,\ldots,|y_n|),
    \\
\log y
    & = &
    (\log y_1,\ldots,\log y_n),
    \quad\text{provided that $y_1,\ldots,y_n>0$,}
    \\
y^\alpha
    & = &
    y_{1}^{\alpha_1}\cdots y_{n}^{\alpha_n},
    \quad\text{provided that this is defined,}
    \\
|\alpha|
    & = &
    \alpha_1+\cdots+\alpha_n,
    \\
\supp(\alpha)
    & = &
    \{i\in \{1,\ldots,n\} : \alpha_i\neq 0\},
    \quad\text{which is called the {\bf\emph{support}} of $\alpha$.}
\end{eqnarray*}
\end{notation}

There is a conflict of notation between this use of $|y|$ and $|\alpha|$, but the context will always distinguish the meaning: if $\alpha$ is a tuple of exponents of a tuple of real numbers, then $|\alpha|$ means $\alpha_1+\cdots+\alpha_n$; if $y$ is a tuple of real numbers not used as exponents, then $|y|$ means $(|y_1|,\ldots,|y_n|)$.  These notations may be combined, such as with $|y|^\alpha = |y_1|^{\alpha_1}\cdots|y_n|^{\alpha_n}$ and $(\log|y|)^\alpha = (\log|y_1|)^{\alpha_1}\cdots(\log|y_n|)^{\alpha_n}$.

\begin{definitions}\label{def:subPrepFct}
Consider a subanalytic set $A\subset\RR^{m+n}$.  We say that $A$ is {\bf\emph{open over $\RR^m$}} if $A_x$ is open in $\RR^n$ for all $x\in\Pi_m(A)$.

We call a function $\theta = (\theta_1,\ldots,\theta_n):A\to\RR^n$ a {\bf\emph{center for $A$ over $\RR^m$}} if $A$ is open over $\RR^m$, and if for each $i\in\{1,\ldots,n\}$ the component $\theta_i$ is an analytic subanalytic function $\theta_i:\Pi_{m+i-1}(A)\to\RR$ with the following two properties.
\begin{enumerate}{\setlength{\itemsep}{3pt}
\item
The range of $\theta_i$ is contained in either $(-\infty,0)$, $\{0\}$ or $(0,\infty)$.  And, when $\theta_i$ is nonzero, the closure of the set $\{y_i/\theta_i(x,y_{<i}) : (x,y)\in A\}$ is a compact subset of $(0,\infty)$.

\item
Let $\tld{y}_i = y_i - \theta_i(x,y_{<i})$.  The set $\{\tld{y}_i : (x,y)\in A\}$ is a subset of either $(-\infty,-1)$, $(-1,0)$, $(0,1)$ or $(1,\infty)$.
}\end{enumerate}
We call $(x,\tld{y}) := (x,\tld{y}_1,\ldots,\tld{y}_n)$ the {\bf\emph{coordinates on $A$ with center $\theta$}}.

A {\bf\emph{rational monomial map on $A$ over $\RR^m$ with center $\theta$}} is a bounded function $\varphi:A\to\RR^M$ of the form
\begin{equation}\label{eq:RatMonMap}
\varphi(x,y) = \left(c_1(x)|\tld{y}|^{\gamma_1}, \ldots, c_M(x)|\tld{y}|^{\gamma_M}\right),
\end{equation}
where $c_1,\ldots,c_M$ are positively-valued analytic subanalytic functions on $\Pi_m(A)$ and $\gamma_1,\ldots,\gamma_M$ are tuples in $\QQ^n$.  Note that $\varphi(A)\subset(0,\infty)^M$.  If $A\subset\RR^m\times(0,1)^n$ and $\theta = 0$, we say that $\varphi$ is {\bf\emph{basic}}.

An analytic function is called a {\bf\emph{unit}} if its range is contained in either $(-\infty,0)$ or $(0,\infty)$.  A function $f:A\to\RR$ is called a {\bf\emph{$\varphi$-function}} if $f = F\circ\varphi$ for some analytic function $F$ whose domain is the closure of the range of $\varphi$; if $F$ is also a unit, then we call $f$ a {\bf\emph{$\varphi$-unit}}.\footnote{A $\varphi$-function was called a ``strong function'' by the authors in \cite[Definition 3.3]{CluckersMiller2}, but there we unintentionally gave an incorrect definition that only required $F$ to be defined on the range of $\varphi$, rather than the closure of the range of $\varphi$.  Unlike in \cite{CluckersMiller2}, here we do not require $F$ to be represented by a single convergent power series.}

A function $f:A\to\RR$ is {\bf\emph{$\varphi$-prepared}} if
\[
f(x,y) = g(x)|\tld{y}|^\alpha u(x,y)
\]
on $A$ for some analytic subanalytic function $g$, tuple $\alpha\in\QQ^n$ and $\varphi$-unit $u$.
\end{definitions}

\begin{definition}\label{def:basicRatMonMap}
To any rational monomial map $\varphi:A\to\RR^M$ over $\RR^m$ with center $\theta$, we associate a basic rational monomial map over $\RR^m$, denoted by $\varphi_\theta$, as follows.  For each $i\in\{1,\ldots,n\}$, the set $\{\tld{y}_i : (x,y)\in A\}$ is contained in either $(-\infty,-1)$, $(-1,0)$, $(0,1)$ or $(1,\infty)$, so there exist unique $\varepsilon_i,\zeta_i\in\{-1,1\}$ such that $0 < \varepsilon_i \tld{y}_{i}^{\,\zeta_i} < 1$ for all $(x,y)\in A$.  Define an analytic isomorphism $T_\theta : A \to A_\theta$ by
\[
T_\theta(x,y) = \left(x, \varepsilon_1 \tld{y}_{1}^{\,\zeta_1}, \ldots, \varepsilon_n \tld{y}_{n}^{\,\zeta_n}\right).
\]
Define $\varphi_\theta := \varphi\circ T^{-1}_{\theta} :A_\theta\to\RR^M$.
\end{definition}

\begin{notation}\label{n:Gamma}
Write $\varphi_\theta(x,y) = \left(c_1(x) y^{\gamma_1}, \ldots, c_M(x) y^{\gamma_M}\right)$ for some $\gamma_1,\ldots,\gamma_M\in\QQ^n$. For each $i\in\{0,\ldots,n\}$, define $\varphi_{\theta,i}$ to be the function on $\Pi_{m+i}(A)$ consisting of the components $c_j(x) y^{\gamma_j}$ of $\varphi_\theta$ such that $\supp(\gamma_j)\subset\{1,\ldots,i\}$, and when $i>0$, such that $i\in\supp(\gamma_j)$.  Thus
\[
\varphi_\theta(x,y) = (\varphi_{\theta,0}(x), \varphi_{\theta,1}(x,y_1),\ldots,\varphi_{\theta,n}(x,y_1,\ldots,y_n)).
\]
For each $i\in\{0,\ldots,n\}$ and $\Box\in\{<,\leq, >, \geq\}$, define $\varphi_{\theta,\Box i} = (\varphi_{\theta,j})_{j\Box i}$ on its appropriate domain.  For example, $\varphi_{\theta,\leq i}$ is the function on $\Pi_{m+i}(A)$ given by
\[
\varphi_{\theta,\leq i}(x,y_{\leq i}) = (\varphi_{\theta,0}(x), \varphi_{\theta,1}(x,y_1),\ldots,\varphi_{\theta,i}(x,y_{\leq i})).
\]
\end{notation}

\begin{definition}
If $C\subset\RR^{m+n}$ is a cell over $\RR^m$, then there exists a unique increasing map $\lambda:\{1,\ldots,d\}\to\{1,\ldots,n\}$ whose image consists of the set of all $i\in\{1,\ldots,n\}$ for which $\Pi_{m+i}(C)$ is of the form \eqref{eq:cell}.  We call $C$ a {\bf\emph{$\lambda$-cell}}.
\end{definition}

Note that $\Pi_{m,\lambda}$ defines an analytic isomorphism from a $\lambda$-cell $C$ onto $\Pi_{m,\lambda}(C)$, and $\Pi_{m,\lambda}(C)$ is a cell over $\RR^m$ that is open over $\RR^m$.

\begin{definition}\label{def:RatMonMapPrep}
We say that $\varphi$ is {\bf\emph{prepared over $\RR^m$}} if $A$ is a cell over $\RR^m$ such that for each $i\in\{1,\ldots,n\}$, if we write
\[
\Pi_{m+i}(A_\theta) = \{(x,y_{\leq i}) : (x,y_{<i})\in\Pi_{m+i-1}(A_\theta), a_i(x,y_{<i}) < y_i < b_i(x,y_{<i})\},
\]
then the functions $a_i$, $b_i$ and $b_i-a_i$ are $\varphi_{\theta,<i}$-prepared, and $a_i$ is either identically zero or is strictly positively-valued.
\end{definition}

\begin{proposition}[Subanalytic Preparation]\label{prop:subPrep}
Suppose that $\F$ is a finite set of subanalytic functions on a subanalytic set $D\subset\RR^{m+n}$.  Then there exists a finite partition $\A$ of $D$ into cells over $\RR^m$ such that for each $A\in\A$, if $A$ is a $\lambda$-cell over $\RR^m$ and we write $g:\Pi_{m,\lambda}(A)\to A$ for the inverse of the projection $\Pi_{m,\lambda}\Restr{A}:A\to\Pi_{m,\lambda}(A)$, then there exists a prepared rational monomial map $\varphi:\Pi_{m,\lambda}(A)\to\RR^M$ over $\RR^m$ such that $f\circ g$ is $\varphi$-prepared for each $f\in\F$.
\end{proposition}

\begin{proof}
This follows from the subanalytic preparation theorem (see \cite{LR97} or \cite{DJMprep}) by induction on $n$.
\end{proof}

\begin{corollary}\label{cor:subPrep}
Suppose that $\F$ is a finite set of constructible functions on a subanalytic set $D\subset\RR^{m+n}$.  Then there exists a finite partition $\A$ of $D$ into cells over $\RR^m$ such that for each $A\in\A$ and $f\in\F$, the restriction of $f$ to $A$ is analytic.  Moreover, if each function in $\F$ is subanalytic, then $\A$ can be chosen so that $f(A)$ is contained in either $(-\infty,0)$, $\{0\}$ or $(0,\infty)$ for each $A\in\A$ and $f\in\F$.
\end{corollary}

\begin{proof}
When $\F$ consists entirely of subanalytic functions, this follows directly from Proposition \ref{prop:subPrep}.  In the general constructible case, fix a finite set $\F'$ of subanalytic functions such that each function in $\F$ is a sum of products of functions of the form $(x,y)\mapsto f(x,y)$ and $(x,y)\mapsto \log g(x,y)$ with $f,g\in\F'$.  Now apply the result of the subanalytic case to $\F'$.
\end{proof}

\begin{definition}\label{def:compatible}
If $\S$ is a set of subsets of a set $X$, we say that a partition $\A$ of $X$ is {\bf\emph{compatible}} with $\S$ if for each $A\in\A$ and each $S\in\S$, either $A\subset S$ or $A\subset X\setminus S$.
\end{definition}

Note that in Proposition \ref{prop:subPrep} and Corollary \ref{cor:subPrep}, the partition $\A$ can be made to be compatible with any prior given finite set of subanalytic subsets of $D$.

%% file: cor.tex
\section{Consequences of the Theorem on Diagrams of Lebesgue Classes}
\label{s:cor}

Throughout this section we use the notation of Theorem \ref{thm:LC}.

\subsection*{Corollaries of the Theorem on Diagrams of Lebesgue Classes}

\begin{corollary}\label{cor:LC=I}
For each $I\in\I$,
\begin{equation}\label{eq:LC=I}
\{x\in E : \LC(f,|\mu|^q,x) = I\} = \left\{x\in E : (g_I(x) = 0)\wedge
\left(\bigwedge_{J\in\I_I} g_J(x)\neq 0\right)\right\},
\end{equation}
where $\I_I = \{J\in\I : I\subsetneq J\}$.
\end{corollary}

\begin{proof}
This follows from \eqref{eq:LC} and from the fact that for each $x\in E$, $\LC(f,\mu,x) = I$ if and only if $I\subset\LC(f,\mu,x)$ and $J\not\subset\LC(f,\mu,x)$ for all $J\in\I_I$.
\end{proof}

The final sentence of Theorem \ref{thm:LC} shows that when $f$ is subanalytic, so is the set \eqref{eq:LC=I}.

\begin{remark}\label{rmk:LCunion}
The set $\LC(f,|\mu|^q,E)$ can be expressed as the disjoint union
\begin{equation}\label{eq:LCdisjointUnion}
\bigcup_{I\in\I}\left(\{x\in E : \LC(f,|\mu|^q,x) = I\} \times I\right)
\end{equation}
and as the (not necessarily disjoint) union
\begin{equation}\label{eq:LCunion}
\bigcup_{I\in\I}\left(\{x\in E : I\subset\LC(f,|\mu|^q,x)\} \times I\right).
\end{equation}
\end{remark}

\begin{proof}
The fact that $\LC(f,|\mu|^q,E)$ equals \eqref{eq:LCdisjointUnion}, and that \eqref{eq:LCdisjointUnion} is contained in \eqref{eq:LCunion}, are both clear.  To see that \eqref{eq:LCunion} is contained in \eqref{eq:LCdisjointUnion}, note that if $(x,p)$ is such that $I\subset\LC(f,|\mu|^q,x)$ and $p\in I$, then $J = \LC(f,|\mu|^q,x)$ and $p\in J$ for some $J\in\I$ with $I\subset J$.
\end{proof}

Observe that \eqref{eq:LC=I} and \eqref{eq:LC} show how to use the functions $\{g_I\}_{I\in\I}$ to define the sets occurring in \eqref{eq:LCdisjointUnion} and \eqref{eq:LCunion}.

\begin{corollary}\label{cor:G_P}
For each $P\subset(0,\infty]$ there exists $G_P\in\C(E)$ such that
\begin{equation}\label{eq:G_P}
\{x\in E : P\subset\LC(f,|\mu|^q,x)\} = \{x\in E : G_P(x) = 0\}.
\end{equation}
\end{corollary}

\begin{proof}
Define $G_P$ to be the product of the $g_I$ for all $I\in\I$ with $P\subset I$.  Then \eqref{eq:G_P} follows from \eqref{eq:LC} and from the fact that for each $x\in E$, we have $P\subset\LC(f,|\mu|^q,x)$ if and only if $\LC(f,|\mu|^q,x) = I$ for some $I\in\I$ with $P\subset I$.
\end{proof}

For each $p\in(0,\infty]$, taking $P = \{p\}$ in \eqref{eq:G_P} shows that $\Int^p(f,|\mu|^q,E)$ is the zero locus of a constructible function.  A very elementary proof of this fact is given in \cite{CluckersMiller2} for the special case when $\mu=1$, $p=1$ and $n=1$.

\begin{corollary}\label{cor:IntpFinite}
The set $\{\Int^p(f,|\mu|^q,E) : p\in(0,\infty]\}$ is finite.
\end{corollary}

\begin{proof}
Since $\I$ is finite by Theorem \ref{thm:LC}, we may fix a finite partition $\J$ of $(0,\infty]$ compatible with $\I$.  If $J\in\J$ and $p\in J$, then for each $I\in\I$, $p\in I$ if and only if $J\subset I$; so $\Int^p(f,|\mu|^q,E) = \{x\in E : J\subset\LC(f,|\mu|^q,x)\}$.  Therefore
\[
\{\Int^p(f,|\mu|^q,E) : p\in(0,\infty]\} = \{\{x\in E : J\subset\LC(f,|\mu|^q,x)\} : J\in\J\},
\]
which is finite because $\J$ is finite.
\end{proof}

\begin{corollary}\label{cor:bdd}
There exists $g\in\C(E)$ such that
\[
\{x\in E : \text{$f(x,\cdot)$ is bounded on $D_x$}\} = \{x\in E  : g(x) = 0\}.
\]
\end{corollary}

\begin{proof}
Zero loci of constructible functions are closed under intersections and unions (by taking sums of squares and by taking products, respectively), so we may assume by Corollary \ref{cor:subPrep} that $D$ is a cell over $\RR^m$ and that $f$ is analytic.  By projecting into a lower dimensional space, we may further assume that $D$ is open over $\RR^m$.  Thus $f(x,\cdot)$ is bounded on $D_x$ if and only if it is in $L^\infty(D_x)$, so we are done by applying Corollary \ref{cor:G_P} with $P=\{\infty\}$.
\end{proof}

Although we will use Theorem \ref{thm:vanish} to prove Theorem \ref{thm:LC}, it is interesting to observe that, conversely, Theorem \ref{thm:vanish} also follows from Theorem \ref{thm:LC}, as follows.

\begin{corollary}\label{cor:idenZero}
There exist $g,h\in\C(E)$ such that
\[
\{x\in E : \text{$f(x,y) = 0$ for all $y\in D_x$}\}
=
\{x\in E  : g(x) = 0\}
\]
and
\[
\{x\in E : \text{$f(x,y) = 0$ for $|\mu|_x$-almost all $y\in D_x$}\}
=
\{x\in E  : h(x) = 0\}.
\]
\end{corollary}

\begin{proof}
Define $F:D\times\RR\to\RR$ by $F(x,y,z) = z f(x,y)$.  Note that for each $x\in E$, $f(x,y) = 0$ for all $y\in D_x$ if and only if $(y,z)\mapsto F(x,y,z)$ is bounded on $D_x\times\RR$, and that $f(x,y) = 0$ for $|\mu|_x$-almost all $y\in D_x$ if and only if $(y,z)\mapsto F(x,y,z)$ is in $L^\infty(\nu_x)$, where $\nu:D\times\RR\to[0,\infty)$ is defined by $\nu(x,y,z) = |\mu(x,y)|$.  So we are done by applying Corollaries \ref{cor:bdd} and \ref{cor:G_P} (with $P = \{\infty\}$) to $F$.
\end{proof}

The following result generalizes \cite[Theorem $1.4'$]{CluckersMiller1}.

\begin{corollary}\label{cor:dense}
Let $q > 0$, $P\subset(0,\infty]$, and $F,\nu\in\C(X\times Y\times\RR^k)$ for some subanalytic sets $X$ and $Y$.  Suppose that for each $x\in X$, the set $\{y\in Y : P\subset\LC(F,|\nu|^q,(x,y))\}$ is dense in $Y$.  Then there exists a subanalytic set $C\subset X\times Y$ such that $C\times P\subset\LC(F,|\nu|^q,X\times Y)$ and $C_x$ is dense in $Y$ for each $x\in X$.
\end{corollary}

\begin{proof}
Assume that $X\subset\RR^m$.  We may assume that $Y = \RR^n$ because the case of a general subanalytic set $Y$ follows from this special case by arguing as in the second paragraph of the proof of \cite[Theorem $1.4'$]{CluckersMiller1}.  By Corollary \ref{cor:G_P} we may fix $g\in\C(X\times\RR^n)$ such that
\begin{equation}\label{eq:LC_P}
\{(x,y)\in X\times\RR^n : P\subset\LC(F,|\nu|^q,(x,y))\}
=
\{(x,y)\in X\times\RR^n : g(x,y) = 0\}.
\end{equation}
By Corollary \ref{cor:subPrep} we may fix a partition $\A$ of $X\times\RR^n$ into subanalytic cells over $\RR^m$ such that $g$ restricts to an analytic function on each $A\in\A$.  Let $C$ be the union of the members of $\A$ that are open over $\RR^m$. Then $C$ is subanalytic, $\Pi_m(C) = X$, and $C_x$ is open and dense in $\RR^n$ for each $x\in X$.  If there exists $(a,b)\in C$ such that $g(a,b)\neq 0$, then $\{y\in C_a : g(a,y) = 0\}$ would be a proper analytic subset of the open set $C_a$, so $\{y\in\RR^n : g(a,y) = 0\}$ would not be dense in $\RR^n$, contradicting \eqref{eq:LC_P} and our assumption on $F$ and $|\nu|^q$.  Therefore $g(x,y) = 0$ for all $(x,y)\in C$, which by \eqref{eq:LC_P} proves the corollary.
\end{proof}

\subsection*{Variants of the Theorem on Diagrams of Lebesgue Classes}
We now show how Theorem \ref{thm:LC} adapts easily to the study of local integrability, complex measures, and measures defined from constructible differential forms on subanalytic sets.  We only discuss the analogs of Theorem \ref{thm:LC} itself, but it follows that analogs of the previous list of corollaries of this theorem hold as well, via the same proofs.

Suppose that $Y\subset\RR^n$ and $f:Y\to\RR$ are Lebesgue measurable, that $\nu$ is a positive measure on $Y$ that is absolutely continuous with respect to the $n$-dimensional Lebesgue measure, and that $p\in(0,\infty]$.  We say that $f$ is {\bf\emph{locally in $L^p(\nu)$}}, written as $f\in L^{p}_{\loc}(\nu)$, if for each $y\in Y$ there exists a neighborhood $U$ of $y$ in $Y$ such that $f\Restr{U}$ is in $L^{p}(\nu\Restr{U})$.  Similarly, we say that $f$ is {\bf\emph{locally bounded on $Y$}} if for each $y\in Y$ there exists a neighborhood $U$ of $y$ in $Y$ such that $f(U)$ is bounded.

For measurable functions $f:D\to\RR$ and $\nu:D\to[0,\infty)$, where $D\subset\RR^{m+n}$ and $E = \Pi_m(E)$, define the sets $\LC_{\loc}(f,\nu,E)$, $\LC_{\loc}(f,\nu,x)$ and $\Int^{p}_{\loc}(f,\nu,E)$ analogously to how $\LC(f,\nu,E)$, $\LC(f,\nu,x)$ and $\Int^p(f,\nu,E)$ were defined in Section \ref{s:mainResults}, but replacing the condition $f(x,\cdot)\in L^p(\nu_x)$ with $f(x,\cdot)\in L^{p}_{\loc}(\nu_{x})$.

\begin{proposition}\label{prop:LClocal}
The local analog of Theorem \ref{thm:LC} holds, which describes the structure of $\LC_{\loc}(f,|\mu|^q,E)$ rather than $\LC(f,|\mu|^q,E)$.
\end{proposition}

\begin{proof}
By extending $f$ and $\mu$ by $0$ on $(E\times\RR^n)\setminus D$, we may assume that $D = E\times\RR^n$.  Define functions $F$ and $\nu$ on $E\times\RR^n\times[-1,1]^n$ by $F(x,y,z) = f(x,y+z)$ and $\nu(x,y,z) = |\mu(x,y+z)|^q$.  The compactness of $[-1,1]^n$ implies that for each $x\in E$ and $p\in(0,\infty]$, $f(x,\cdot)\in L_{\loc}^{p}(|\mu|^{q}_{x})$ if and only if $F(x,y,\cdot)\in L^p(\nu_{(x,y)})$ for all $y\in\RR^n$.  Therefore
\[
\LC_{\loc}(f,|\mu|^q,x) = \bigcap_{y\in\RR^n} \LC(F,\nu,(x,y)).
\]
Theorem \ref{thm:LC} shows that $\{\LC(F,\nu,(x,y)) : (x,y)\in E\times\RR^n\}$ is a finite set of subintervals of $(0,\infty]$ with endpoints in $(\Span_{\QQ}\{1,q\}\cap[0,\infty))\cup\{\infty\}$, so the set
\[
\I_{\loc} := \{\LC_{\loc}(f,|\mu|^q,x) : x\in E\}
\]
is of this form as well.  Let $I\in\I_{\loc}$.  By Corollary \ref{cor:G_P} we may fix $g\in\C(E\times\RR^n)$ such that
\[
\{(x,y)\in E\times\RR^n : I\subset\LC(F,\nu,(x,y))\} = \{(x,y)\in E\times\RR^n : g(x,y)=0\}.
\]
Thus
\[
\{x\in E : I \subset\LC_{\loc}(f,|\mu|^q,x)\}
=
\{x\in E : \text{$g(x,y) = 0$ for all $y\in\RR^n$}\},
\]
and this set is the zero locus of a constructible function by Theorem \ref{thm:vanish} (or Corollary \ref{cor:idenZero}).
\end{proof}

Suppose that $f$ and $\nu$ are complex-valued Lebesgue measurable functions on a measurable set $D\subset\RR^{m+n}$ such that $\nu(x,\cdot)$ is Lebesgue integrable on $D_x$ for all $x\in E$, where $E =\Pi_m(D)$.  For each $x\in E$,  define a complex measure $\nu_x$ on $D_x$ by setting
\[
\nu_x(Y) = \int_Y \nu(x,y) dy
\]
for each Lebesgue measurable set $Y\subset D_x$.  The notion of an $L^p$-class with respect to a complex-measure is defined using the absolute variation of the measure, so we define $\LC(f,\nu,E) := \LC(|f|,|\nu|,E)$, $\LC(f,\nu,x) := \LC(|f|,|\nu|,x)$ for each $x\in E$, and $\Int^p(f,\nu,E) := \Int^p(|f|,|\nu|,E)$ for each $p\in(0,\infty]$.

\begin{proposition}\label{prop:LCcomplex}
The complex analog of Theorem \ref{thm:LC} holds with $q=1$, which describes the structure of $\LC(f,\mu,E)$ for complex-valued functions $f$ and $\mu$ on a subanalytic set $D\subset\RR^{m+n}$ whose real and imaginary parts are constructible, where $\mu(x,\cdot)$ is Lebesgue integrable on $D_x$ for all $x$ in $E = \Pi_m(D)$.
\end{proposition}

\begin{proof}
Apply Theorem \ref{thm:LC} to the constructible functions $|f|^2$ and $|\mu|^2$ with $q=1/2$.  Then note that for any $p\in (0,\infty]$, $|f|\in L^p(|\mu|_x)$ if and only if $|f|^2 \in L^{p/2}(|\mu|_x)$.
\end{proof}

For the last result of this section, consider a subanalytic set $D\subset\RR^{m+n}$ such that for each $x$ in $E := \Pi_m(D)$, the fiber $D_x$ is a smooth $k$-dimensional submanifold of $\RR^n$. For each $x\in E$, consider a smooth $k$-form $\omega_x$ on $D_x$, such that moreover there exist constructible functions $\omega_{i_1,\ldots,i_k}(x,y)$ on $D$ with $1\leq i_1 < \cdots < i_k\leq n$ such that
\[
\omega_x(y) = \sum_{1\leq i_1 < \cdots < i_k\leq n} \omega_{i_1,\ldots,i_k}(x,y) dy_{i_1} \wedge \cdots \wedge dy_{i_k}.
\]
For each $x\in E$, write $|\omega_x|$ for the measure on $D_x$ associated to the smooth $k$-form $\omega_x$.
For $f\in\C(D)$, consider
\[
\LC(f,\omega_x,x) = \{p\in(0,\infty] : f(x,\cdot)\in L^p(|\omega_x|)\},
\]
and
\[
\LC(f,\omega,E) = \{(x,p)\in E\times(0,\infty] : f(x,\cdot) \in L^p(|\omega_x|)\},
\]
where $\omega$ stands for the family  $(\omega_x)_{x\in E}$.

\begin{proposition}\label{prop:LCdiffForm} With the above notation for $D$, $\omega$, and $E$, and with $f\in \C(D)$,
the analog of Theorem \ref{thm:LC} holds for $\LC(f,\omega,E)$.
To adapt the last sentence of Theorem \ref{thm:LC} to $\LC(f,\omega,E)$, the extra assumption that $\mu$ be subanalytic should be replaced by the condition that the $\omega_{i_1,\ldots,i_k}$ be subanalytic.
\end{proposition}

\begin{proof}
Because $D$ is subanalytic, basic o-minimality implies that there exists a finite family $\U$ of subanalytic subsets of $D$ which covers $D$ and is such that the following hold for each $U\in\U$:
\begin{enumerate}
\item
for every $x\in \Pi_m(U)$, the fiber $U_x$ is open in $D_x$;

\item
there exists an increasing function $\lambda^U:\{1,\ldots,k\}\to\{1,\ldots,n\}$ such that for each $x\in\Pi_m(U)$, the projection $\Pi_{\lambda^U}$ is injective on $U_x$ and has constant rank $k$.
\end{enumerate}
For each $U\in\U$, let $G^U(x,z) = (x,g^U(x,z))$ be the inverse of $\Pi_{m,\lambda^U}:U\to\Pi_{m,\lambda^U}(U)$, where $z = (z_1,\ldots,z_k)$.  Then for each $U\in\U$, the functions $f\circ G^U$ and
\[
\omega^U(x,z) := \sum_{1\leq i_1 < \cdots < i_k\leq n} \omega_{i_1,\ldots,i_k}(x,g^U(x,z)) \frac{\partial(g_{i_1}^{U},\ldots,g_{i_k}^{U})}{\partial(z_1,\ldots,z_k)}(x,z)
\]
are both constructible functions on $U$, and in the case that $f$ and all the $\omega_{i_1,\ldots,i_k}$ are subanalytic, the $\omega^U$ and $f\circ G^U$ also are. Hence, Theorem \ref{thm:LC} applies to $\LC(f\circ G^U,|\omega^U|,\Pi_m(U))$.  The proposition now follows relatively easily from this and from the fact that
\[
\LC(f\Restr{U},\omega\Restr{U},\Pi_m(U)) = \LC(f\circ G^U,|\omega^U|, \Pi_m(U))
\]
for each $U\in\U$.
\end{proof}

%% file: vanish.tex
\section{Fiberwise Vanishing of Constructible Functions}
\label{s:vanish}

This section proves Theorem \ref{thm:vanish}.

\begin{proof}[Proof of Theorem \ref{thm:vanish}]
Let $f\in\C(D)$ for a subanalytic set $D\subset\RR^{m+n}$, and put $E = \Pi_m(D)$. Write $V = \{x\in E : \text{$f(x,y) = 0$ for all $y\in D_x$}\}$.  We proceed by induction on $n$.

First suppose that $n=1$.  By Corollary \ref{cor:subPrep} we may fix a finite partition $\A$ of $D$ into cells over $\RR^m$ such that the restriction of $f$ to $A$ is analytic for each $A\in\A$.  We claim that for each $A\in\A$ there exists $g_A\in\C(\Pi_m(A))$ such that
\[
\{x\in\Pi_m(A) : \text{$f(x,y) = 0$ for all $y\in A_x$}\}
=
\{x\in\Pi_m(A) : g_A(x) = 0\}.
\]
The theorem (with $n=1$) follows from the claim, for then
\[
V
=
\left\{x\in E : \sum_{A\in\A} (g'_A(x))^2 = 0\right\},
\]
where $g'_A:E\to\RR$ is the extension of $g_A$ by $0$ on $E\setminus\Pi_m(A)$.  To prove the claim, fix $A\in\A$.  We may assume that $A$ is open over $\RR^m$, else the claim is trivial.  Since $f(x,\cdot)$ is analytic on $A_x$ for each $x\in\Pi_m(A)$, and since $f\Restr{A}$ is definable in the expansion of the real field by all restricted analytic functions and the exponential function, which is o-minimal (see Van den Dries, Macintyre and Marker \cite{vdDMM}, or Lion and Rolin \cite{LR97}), it follows that we may fix a positive integer $N$ such that for each $x\in\Pi_m(A)$, $f(x,y) = 0$ for all $y\in A_x$ if and only if there exist distinct $y_1,\ldots,y_N\in A_x$ such that $f(x,y_1) = \cdots = f(x,y_N) = 0$.  So fix subanalytic functions $\xi_1,\ldots,\xi_N :\Pi_m(A)\to\RR$ whose graphs are disjoint subsets of $A$.  Then the claim holds for the function
\[
g_A(x) = \sum_{i=1}^{N} \left(f(x,\xi_i(x))\right)^2.
\]
This establishes the theorem when $n = 1$.

Now suppose that $n>1$, and inductively assume the theorem holds with $k$ in place of $n$ for each $k < n$.  The set $V$ is defined by the formula
\[
(x\in E) \wedge \forall y\in\RR^n((x,y)\in D \implies f(x,y) = 0).
\]
Applying the induction hypothesis twice shows that that this formula is equivalent to
\[
(x\in E) \wedge \forall y_1\in\RR((x,y_1)\in\Pi_{m+1}(D) \implies h(x,y_1) = 0)
\]
for some $h\in\C(\Pi_{m+1}(D))$, which in turn is equivalent to
\[
(x\in E) \wedge (g(x) = 0)
\]
for some $g\in\C(E)$.  Thus $V = \{x\in E : g(x) = 0\}$.
\end{proof}

%% file: subRect.tex
\section{Parameterized Rectilinearization of Subanalytic Functions}\label{s:subRect}

\begin{definition}\label{def:rectRatMonMap}
Consider $l\in\{0,\ldots,n\}$ and a rational monomial map $\psi$ on $B$ over $\RR^m$, where $B\subset\RR^{m+n}$.  We say that $\psi$ is {\bf\emph{$l$-rectilinear over $\RR^m$}} if $B$ is $l$-rectilinear over $\RR^m$ (as defined prior to Theorem \ref{thm:subRect}) and if $\psi$ is of the form
\[
\psi(x,y) = \left(c_1(x) y_{\leq l}^{\gamma_1}, \ldots, c_N(x) y_{\leq l}^{\gamma_N}, y_{l+1},\ldots,y_n\right)
\]
for some positively-valued analytic subanalytic functions $c_1,\ldots,c_N$ on $\Pi_m(B)$ and tuples $\gamma_1,\ldots,\gamma_N$ in $\QQ^l$.  We say that set $B$, or a rational monomial map $\psi$ on $B$ over $\RR^m$, is {\bf\emph{rectilinear over $\RR^m$}} to mean that it is $\l$-rectilinear over $\RR^m$ for some $l$.
\end{definition}

\begin{definition}\label{def:openPart}
For a subanalytic set $D\subset\RR^{m+n}$, an {\bf\emph{open partition of $D$ over $\RR^m$}} is a finite family $\A$ of disjoint subanalytic subsets of $D$ that are open over $\RR^m$ and are such that $\dim(D\setminus\bigcup\A)_x < n$ for all $x\in\Pi_m(D)$.
\end{definition}

The main purpose of this section is to prove the following proposition.

\begin{proposition}\label{prop:subRect}
Let $\F$ be a finite set of subanalytic functions on a subanalytic set $D\subset\RR^{m+n}$.  Then there exists an open partition $\A$ of $D$ over $\RR^m$ such that for each $A\in\A$ there exists a subanalytic analytic isomorphism $F:B\to A$ over $\RR^m$ with $B\subset\RR^{m+n}$, and there exist rational monomial maps $\varphi$ on $A$ and $\psi$ on $B$ over $\RR^m$ with the following properties.
\begin{enumerate}
\item
Pullback Property: Each function in $\{f\circ F\}_{f\in\F}\cup\{\det\PD{}{F}{y}\}$ is $\psi$-prepared, and $\psi$ is rectilinear over $\RR^m$.

\item
Pushforward Property:
The components of  $F^{-1}$ are $\varphi$-prepared, and $\psi\circ F^{-1}$ is a $\varphi$-function.
\end{enumerate}
\end{proposition}

The purpose of the pushforward property is that it ensures that for each subanalytic function $h:B\to\RR$ that is $\psi$-prepared, $h\circ F^{-1}$ is $\varphi$-prepared.   This proposition is essentially Theorem \ref{thm:subRect}, the only differences being that the theorem does not mention the pushforward property and that the theorem deals with an actual partition of $D$ rather than just an open partition of $D$ over $\RR^m$.  In the proposition we use open partitions over $\RR^n$, rather than actual partitions, because it allows the proof of the proposition to be stated somewhat more simply since we may ignore subsets of $D$ whose fibers over $\RR^m$ have dimension less than $n$, and doing so is of no loss to the study of $L^p$-spaces on $D_x$.

Before proving the proposition, we use it to prove Theorem \ref{thm:subRect} and also Theorem \ref{thm:LC} when $f$ and $\mu$ are assumed to be subanalytic.

\begin{proof}[Proof of Theorem \ref{thm:subRect}]
Let $\F$ be a finite set of subanalytic functions on a subanalytic set $D\subset\RR^{m+n}$.  We proceed by induction on $n$.  The base case of $n=0$ is trivial, so assume that $n>0$ and that the theorem holds with $k$ in place of $n$ for all $k < n$.  Let $\A$ be the open partition of $D$ over $\RR^m$ given by applying Proposition \ref{prop:subRect} to $\F$, and let $D' = \bigcup\A$.  Thus the theorem holds for $\F\Restr{D'}$.  It follows from the induction hypothesis that the theorem also holds for $\F\Restr{D\setminus D'}$, since $D\setminus D'$ may be partitioned into cells over $\RR^m$, and each of these cells projects via an analytic isomorphism into $\RR^{m+d}$ for some $d < n$.
\end{proof}

The following lemma of one-variable calculus, and its corollary, are apparent.

\begin{lemma}\label{lemma:intBdd}
Let $\alpha\in\RR$ and $\beta\geq 0$.  Then the function $t\mapsto t^\alpha (\log t)^\beta$ is
\begin{enumerate}{\setlength{\itemsep}{3pt}
\item
integrable on $(0,1)$ if and only if $\alpha > -1$;

\item
bounded on $(0,1)$ if and only if $\alpha > 0$ or $\alpha=\beta=0$.
}\end{enumerate}
\end{lemma}

\begin{corollary}\label{cor:rectIntBdd}
Suppose that $A\subset\RR^n$ is $l$-rectilinear over $\RR^0$, and let $\alpha = (\alpha_1,\ldots,\alpha_n)\in\RR^n$ and $\beta = (\beta_1,\ldots,\beta_n)\in[0,\infty)^n$.
Then the function $y\mapsto y^\alpha|\log y|^\beta$ is
\begin{enumerate}{\setlength{\itemsep}{3pt}
\item
integrable on $A$ if and only if for all $i\in\{l+1,\ldots,n\}$, $\alpha_i > -1$;

\item
bounded on $A$ if and only if for all $i\in\{l+1,\ldots,n\}$, $\alpha_i > 0$ or $\alpha_i=\beta_i=0$.
}\end{enumerate}
\end{corollary}

Note that if $A\subset\RR^{m+n}$ is $l$-rectilinear over $\RR^m$, then by applying Corollary \ref{cor:rectIntBdd} to each of the fibers $A_x$, we see that $y\mapsto y^\alpha|\log y|^\beta$ is integrable on $A_x$ either for all $x\in\Pi_m(A)$ or for no $x\in\Pi_m(A)$, according to whether the condition given in clause 1 of the corollary holds; and likewise for boundedness and clause 2.

\begin{proof}[Proof of Theorem \ref{thm:LC} in the Subanalytic Case]
Suppose that $q > 0$ and that $f$ and $\mu$ are real-valued subanalytic functions on $D\subset\RR^{m+n}$.  Put $E = \Pi_m(D)$ and $\I = \{\LC(f,|\mu|^q,x) : x\in E\}$.  Apply Proposition \ref{prop:subRect} to $\F = \{f,\mu\}$.  This constructs an open partition $\A$ of $D$ over $\RR^m$ such that for each $A\in\A$, there exist a subanalytic analytic isomorphism $F:B\to A$ over $\RR^m$ and a rectilinear rational monomial map $\psi$ on $B$ over $\RR^m$ such that $f\circ F$, $\mu\circ F$ and $\det\PD{}{F}{y}$ are $\psi$-prepared.

Focus on one $A\in\A$, along with its associated maps $F:B\to A$ and $\psi$ on $B$, where $\psi$ is $l$-rectilinear over $\RR^m$.  Define $\nu:B\to\RR$ by
\[
\nu(x,y) = \left|\mu\circ F(x,y)\right|^q \left|\det\PD{}{F}{y}(x,y)\right|.
\]
On $B$ write
\begin{eqnarray*}
f\circ F(x,y)
    & = &
    a(x)y^\alpha u(x,y),
    \\
\nu(x,y)
    & = &
    b(x)y^\beta v(x,y),
\end{eqnarray*}
for some analytic subanalytic functions $a$ and $b$, tuples $\alpha = (\alpha_1,\ldots,\alpha_n)\in\QQ^n$ and $\beta = (\beta_1,\ldots,\beta_n)\in (\Span_{\QQ}\{1,q\})^n$, and $\psi$-units $u$ and $v$.  We may assume that $a$ and $b$ have constant sign.  If $a = 0$ or $b = 0$, let $I_A = (0,\infty]$.  Otherwise, let $I_A$ be the set consisting of all $p\in(0,\infty)$ such that $\alpha_i p + \beta_i > -1$ for all $i\in\{l+1,\ldots,n\}$, and also consisting of $\infty$ if $\alpha_i\geq 0$ for all $i\in\{l+1,\ldots,n\}$.  Note that $I_A$ is a subinterval of $(0,\infty]$ with endpoints in $(\Span_{\QQ}\{1,q\}\cap[0,\infty))\cup\{\infty\}$.  Also note that by Corollary \ref{cor:rectIntBdd},
\[
\LC(f\Restr{A},|\mu|^q\Restr{A},\Pi_m(A)) = \LC(f\circ F,\nu,\Pi_m(A)) = \Pi_m(A)\times I_A.
\]

Now, for each $x\in E$, the set $\LC(f,\mu,x)$ is a subinterval of $(0,\infty]$ with endpoints in $(\Span_{\QQ}\{1,q\}\cap[0,\infty))\cup\{\infty\}$ because it equals the intersection of the sets $I_A$ for all $A\in\A$ with $x\in\Pi_m(A)$.  This, and the fact that $\A$ is finite, also implies that $\I$ is finite.  To finish, let $I\in\I$, and note that $\{x\in E : I\subset\LC(f,\mu,x)\}$ equals
\[
\{x\in E : \text{$I\subset I_A$ for all $A\in\A$ with $x\in\Pi_m(A)$}\},
\]
which is a subanalytic set, and hence is the zero locus of a subanalytic function.
\end{proof}

We now turn our attention to proving Proposition \ref{prop:subRect}.

\begin{lemma}\label{lemma:rectBddMon}
Let $A\subset\RR^n$ be $l$-rectilinear over $\RR^0$, and let $\alpha = (\alpha_1,\ldots,\alpha_n)\in\QQ^n$.
\begin{enumerate}{\setlength{\itemsep}{3pt}
\item
If $\{y^\alpha : y\in A\}$ is bounded, then $\alpha_{l+1},\ldots,\alpha_n\geq 0$.

\item
Let $\beta\in\QQ$ and $B = \{(y,z)\in A\times\RR : a(y) < z < 1\}$, where $0 \leq a(y) < 1$ for all $y\in A$.  If $\{y^\alpha z^\beta : (y,z)\in B\}$ is bounded, then $\alpha_{l+1},\ldots,\alpha_n\geq 0$.
}\end{enumerate}
\end{lemma}

\begin{proof}
Statement 1 is clear.  Statement 2 follows from statement 1 because $\{y^\alpha : y\in A\}$ is in the closure of the set $\{y^\alpha z^\beta : (y,z)\in B\}$, so $\{y^\alpha : y\in A\}$ is bounded if $\{y^\alpha z^\beta : (y,z)\in B\}$ is bounded.
\end{proof}

The following lemma is apparent.

\begin{lemma}\label{lemma:adjustment}
Let $\varphi:A\to\RR$ be a basic rational monomial map over $\RR^m$, where  $A\subset\RR^{m+n}$ and $\varphi(x,y) = c(x)y^\alpha$.
\begin{enumerate}{\setlength{\itemsep}{3pt}
\item
If $A$ is $l$-rectilinear over $\RR^m$ and $\alpha\in\QQ^l\times\NN^{n-l}$, then $c(x)y_{\leq l}^{\alpha_{\leq l}}$ is bounded on $\Pi_{m+l}(A)$, and $\varphi$ is a $(c(x)y_{\leq l}^{\alpha_{\leq l}}, y_l,\ldots,y_n)$-function.

\item
Let $j\in\{1,\ldots,n\}$, and put $y' = (y_{<j},y_{>j})$ and $\alpha' = (\alpha_{<j},\alpha_{>j})$.  If the closure of $\{y_j : (x,y)\in A\}$ is contained in $(0,1]$, then $c(x)(y')^{\alpha'}$ is bounded on $A$, and $\varphi$ is a $(c(x)(y')^{\alpha'},y_j)$-function.
}\end{enumerate}
\end{lemma}

The proof of Proposition \ref{prop:subRect} will use two types of constructions, called pullback and pushforward constructions, to achieve the desired pullback and pushforward properties.

\begin{definition}\label{def:pullback}
Suppose we are given a basic rational monomial map $\varphi:A\to\RR^M$ over $\RR^m$, where $A\subset\RR^{m+n}$ is a cell over $\RR^m$.  A {\bf\emph{pullback construction for $\varphi$}} consists of a subanalytic map $F:B\to A$ and a basic rational monomial map $\psi:B\to\RR^N$ over $\RR^m$, diagrammed as follows,
\[
\xymatrix{
*+[r]{B} \ar@{.>}[r]^-{F} \ar@{.>}[d]^-{\psi}
    & *+[r]{A} \ar[d]^-{\varphi}
\\
*+[r]{\RR^N}
    & *+[r]{\RR^M\,\, ,}
}
\]
where $B\subset\RR^{m+n}$ is a cell over $\RR^m$, $F:B\to F(B)$ is an analytic isomorphism over $\RR^m$, $\det\PD{}{F}{y}$ and the components of $F$ are $\psi$-prepared, and $\varphi\circ F$ is a $\psi$-function.
\end{definition}

Observe that these properties ensure that if $h$ is any $\varphi$-prepared function, then $h\circ F$ is $\psi$-prepared.

We will use the six types of pullback constructions listed below, where
\begin{equation}\label{eq:prepFormSet}
\Pi_{m+j}(A) = \{(x,y_{\leq j}) : (x,y_{<j})\in\Pi_{m+j-1}(A), a_j(x,y_{<j}) < y_j < b_j(x,y_{<j})\}
\end{equation}
for each $j\in\{1,\ldots,n\}$.  When defining $F$ below, we only specify its action on coordinates on which it acts nontrivially.
\begin{enumerate}{\setlength{\itemsep}{2pt}
\item
\emph{Adjustment}: This means that $F$ is the identity map (but $\psi$ may be different from $\varphi$).

\item
\emph{Restriction}: This means that $F$ is an inclusion map and $\psi = \varphi\Restr{B}$.

\item
\emph{Power Substitution in $y_j$}: This means that $F$ sends $y_j\mapsto y_{j}^{p}$ for some positive integer $p$, and $\psi = \varphi\circ F$.

\item
\emph{Blowup in $y_j$}: This means that we are assuming that $\varphi_{\leq j}$ is prepared over $\RR^{m+j-1}$, that $F$ sends $y_j\mapsto y_jb_j(x,y_{<j})$, and that $\psi$ is the pullback of $\bar{\varphi}$ by the transformation sending $y_j\mapsto y_j \widehat{b}(x) y_{<j}^{\beta}$, where $b_j(x,y_{<j}) = \widehat{b}(x) y_{<j}^{\beta}u(x,y_{<j})$ is the $\varphi_{<j}$-prepared form of $b_j$ and $\bar{\varphi}$ is the natural extension of $\varphi$ to $\Pi_m(A)\times(0,\infty)^n$.

\item
\emph{Flip in $y_j$}: This means we are assuming that $\varphi$ is prepared over $\RR^{m+j-1}$, that the closure of $\{y_j : (x,y)\in A\}$ is contained in $(0,1]$, that $b_j=1$, and that $\varphi$ is of the form
\begin{equation}\label{eq:flipPullback}
\varphi(x,y) = (\varphi_{<j}(x,y_{<j}), y_j, \varphi_{>j}(x,y_{<j},y_{>j}));
\end{equation}
$F$ is the transformation sending $y_j\mapsto 1-y_j$, and $\psi$ is defined by the formula on the right side of \eqref{eq:flipPullback}, but on $B$ rather than on $A$.

\item
\emph{Swap in $y_i$ and $y_j$}:
This means that $F$ is the transformation sending $(y_i,y_j) \mapsto (y_j,y_i)$ and $\psi = \varphi\circ F$, provided that the resulting set $B$ is still a cell over $\RR^m$.
}\end{enumerate}

\begin{remark}\label{rmk:flip}
Note that when $(F,\psi)$ is a flip in $y_j$, we always assume that $\varphi$ is prepared over $\RR^{m+j-1}$ and that the closure of $\{y_j : (x,y)\in A\}$ is contained in $(0,1]$.  We may therefore additionally assume that for each $i\in\{j+1,\ldots,n\}$, the monomials in $y_{<i}$ occurring outside the units in the prepared forms of $a_i$, $b_i$ and $b_i-a_i$ do not contain any nonzero powers of $y_j$, because any nonzero powers of $y_j$ may be included in the units.
\end{remark}

\begin{definition}\label{def:pushforward}
Suppose that we are given a basic rational monomial map $\psi:B\to\RR^N$ over $\RR^m$ and a subanalytic analytic isomorphism $F:B\to A$ over $\RR^m$, where $A,B\subset\RR^{m+n}$.  A {\bf\emph{pushforward construction for $\psi$ and $F$}} is a basic rational monomial map $\varphi:A\to\RR^M$ over $\RR^m$, diagrammed as follows,
\[
\xymatrix{
*+[r]{B} \ar[r]^-{F} \ar[d]^-{\psi}
    & *+[r]{A} \ar@{.>}[d]^-{\varphi}
\\
*+[r]{\RR^N}
    & *+[r]{\RR^M\,\, ,}
}
\]
where the components of $F^{-1}$ are $\varphi$-prepared and $\psi\circ F^{-1}$ is a $\varphi$-function.
\end{definition}

Observe that these properties ensure that if $h$ is any $\psi$-prepared function, then $h\circ F^{-1}$ is $\varphi$-prepared.

If $F:B\to A$ is a map from any one of the six types of pullback constructions described above, $\psi':B'\to\RR^{N'}$ is a basic rational monomial map over $\RR^m$ with $B'\subset B$, and $A' = F(B')$, then the maps $F\Restr{B'}:B'\to A'$ and $\psi'$ have an obvious pushforward construction $\varphi':A'\to\RR^{M'}$, provided that when $F$ is a flip in $y_j$, the map $\psi'$ is of the form $\psi'(x,y) = (\psi'_{<j}(x,y_{<j}),y_j,\psi'_{>j}(x,y_{<j},y_{>j}))$.

\begin{proof}[Proof of Proposition \ref{prop:subRect}]
Let $\F$ be a finite set of subanalytic functions on $D\subset\RR^{m+n}$.  Apply Proposition \ref{prop:subPrep} to $\F$, and focus on one rational monomial map $\varphi:A\to\RR^M$ over $\RR^m$ that this gives for which $A$ is open over $\RR^m$.  Thus $\varphi$ is prepared, and each function in $\F$ restricts to a $\varphi$-prepared function on $A$.  Let $\theta$ be the center of $\varphi$.  We will first construct finitely many sequences of maps diagrammed as follows,
\begin{equation}\label{eq:pullback}
\xymatrix{
B = A_k \ar[r]^-{F_k} \ar@<2ex>[d]^-{\varphi^{[k]} = \psi}
    & *+[r]{A_{k-1}} \ar[r] \ar[d]^-{\varphi^{[k-1]}}
    & \cdots \ar[r]
    & *+[r]{A_1} \ar[r]^-{F_1} \ar[d]^-{\varphi^{[1]}}
    & *+[r]{A_0 = A_\theta} \ar[rr]^-{T_{\theta}^{-1}}
        \ar[d]^-{\varphi^{[0]} = \varphi_\theta}
    &
    & *+[r]{A} \ar[d]^-{\varphi}
\\
\RR^N = \RR^{M_k}
    & *+[r]{\RR^{M_{k-1}}}
    &
    & *+[r]{\RR^{M_1}}
    & *+[r]{\RR^{M_0} = \RR^M}
    &
    & *+[r]{\RR^M\,\, ,}
}
\end{equation}
where for each $i\in\{1,\ldots,k\}$ the maps $F_i$ and $\varphi^{[i]}$ are a pullback construction for $\varphi^{[i-1]}$ of one of the six types listed above, the map $\psi$ is rectilinear over $\RR^m$, and the ranges of the maps $F:B\to A$ given by $F = T^{-1}_{\theta}\circ F_1\circ\cdots\circ F_k$ for all such sequences \eqref{eq:pullback} constructed form an open partition of $A$ over $\RR^m$.  Doing this proves the pullback property.  We will construct \eqref{eq:pullback} to also have the following property.
\begin{equation}\label{eq:flip}
\text{\parbox{5in}{
For each $j\in\{1,\ldots,n\}$, at most one map $F_i$ in \eqref{eq:pullback} is a flip in $y_j$.
}}
\end{equation}

Assuming we can construct \eqref{eq:pullback} as such, to prove the pushforward property it suffices to define $A' = F(B)$, to inductively define $B_k = B$ and $B_{i-1} = F_i(B_i)$ for each $i\in\{1,\ldots,k\}$, and to show that we can construct maps diagrammed as follows,
\begin{equation}\label{eq:pushforward}
\xymatrix{
B = B_k \ar[r]^-{F_k} \ar@<2ex>[d]^-{\psi^{[k]} = \psi}
    & *+[r]{B_{k-1}} \ar[r] \ar[d]^-{\psi^{[k-1]}}
    & \cdots \ar[r]
    & *+[r]{B_1} \ar[r]^-{F_1\Restr{B_1}} \ar[d]^-{\psi^{[1]}}
    & *+[r]{B_0} \ar[r]^-{T_{\theta}^{-1}\Restr{B_0}}
        \ar[d]^-{\psi^{[0]}}
    & *+[r]{A'} \ar[d]^-{\varphi' = \psi^{[0]}\circ T_\theta\Restr{A'}}
\\
\RR^N = \RR^{N_k}
    & *+[r]{\RR^{N_{k-1}}}
    &
    & *+[r]{\RR^{N_1}}
    & *+[r]{\RR^{N_0}}
    & *+[r]{\RR^{M'} = \RR^{N_0}\,\, ,}
}
\end{equation}
where for each $i\in\{1,\ldots,k\}$, $\psi^{[i-1]}$ is a pushforward construction for $F_i\Restr{B_i}:B_i\to B_{i-1}$ and $\psi^{[i]}$.  (Thus the map $\varphi:A\to\RR^M$ in the statement of the theorem is being denoted by $\varphi':A'\to\RR^{M'}$ here in the proof.)
These pushforward constructions will be possible because if a map $F_i$ in \eqref{eq:pullback} is a flip in $y_j$, we can ensure that $\psi^{[i]}$ is of the form \eqref{eq:flipPullback}.  Indeed, from among the six types of pullback and pushforward constructions we use, only blowups in one of the variables $y_j,\ldots,y_n$ can possibly destroy the form \eqref{eq:flipPullback}. So Remark \ref{rmk:flip} and \eqref{eq:flip} imply that, in fact, all the maps $\varphi^{[i]},\ldots,\varphi^{[k]}$ and $\psi^{[k]},\ldots,\psi^{[i]}$ are of the form \eqref{eq:flipPullback}.

So it remains to construct the sequences \eqref{eq:pullback}.  This is done by an induction, and to simplify notation we will write $\varphi:A\to\RR^M$ instead of the more cumbersome $\varphi^{[i]}:A_i\to\RR^{M_i}$.  (So we are now assuming that $\varphi$ is basic.)  Let $d\in\{1,\ldots,n\}$, and inductively assume that $\varphi_{<d}$ is $l$-rectilinear over $\RR^m$ for some $l\in\{0,\ldots,d-1\}$ and that $\varphi$ is prepared over $\RR^{m+d-1}$.   Thus $A$ is a cell over $\RR^m$, so we use the notation \eqref{eq:prepFormSet}.  To complete the construction, it suffices to show that after taking an open partition of $A$ over $\RR^m$ and pulling back $\varphi$, we may reduce to the case that $\varphi_{\leq d}$ is rectilinear and $\varphi$ is prepared over $\RR^{m+d}$.

By pulling back by a blowup in $y_d$ and then by power substitutions in $y_{l+1},\ldots,y_d$, and using Lemma \ref{lemma:rectBddMon}, we may assume that $b_d = 1$ and that all the powers of $y_{l+1},\ldots,y_d$ occurring in the components of $\varphi$ are natural numbers, and when $a_d > 0$, that all the powers of $y_{l+1},\ldots,y_{d-1}$ in the monomials occurring outside the units in the $\varphi_{<d}$-prepared forms of $a_d$ and $1-a_d$ are also natural numbers.  There are two cases that can be handled very easily.
\begin{caselist}
\item
$a_d = 0$.
\\
In this case, $\Pi_{m+d}(A)$ is $l$-rectilinear, so we are done after using Lemma \ref{lemma:adjustment}.1 to adjust $\varphi$.

\item
The closure of $\{y_d : (x,y)\in A\}$ is contained in $(0,1]$.
\\
In this case, use Lemma \ref{lemma:adjustment}.2 to adjust $\varphi$ to assume that $\varphi$ is of the form \eqref{eq:flipPullback}, and then apply a flip in $y_d$ to reduce to Case 1.
\end{caselist}
(Note that if we reduce to either of these two cases, we need not require that $b_d=1$ or that the requisite powers of $y_{l+1},\ldots,y_d$ are natural numbers, because the blowup and power substitutions mentioned just prior to these cases can be applied if needed.)  So assume that $a_d > 0$, and write
\[
a_d(x,y_{<d}) = \widehat{a}(x) y_{<d}^{\alpha} u(x,y_{<d})
\]
for some analytic subanalytic function $\widehat{a}$, tuple of rational numbers $\alpha = (\alpha_1,\ldots,\alpha_{d-1})$, and $\varphi_{<d}$-unit $u$.  We proceed by induction on $|\supp(\alpha_{>l})|$, the cardinality of the set $\supp(\alpha_{>l})$.

Suppose that $\supp(\alpha_{>l})$ is empty, and write $y_{\leq l}^{\alpha}$ instead of $y_{<d}^{\alpha}$.  Fix a constant $C$ that is greater than the supremum of the range of $u$.  Construct a partition of $\Pi_{m+l}(A)$ into cells over $\RR^m$ compatible with the condition $\widehat{a}(x)y_{\leq l}^{\alpha}C = 1$.  By considering the restriction of $\varphi$ to $A\cap(B\times\RR^{n-l})$ for each cell $B$ from this partition that is open over $\RR^m$, we may assume that either $\widehat{a}(x)y_{\leq l}^{\alpha}C > 1$ on $A$ or $\widehat{a}(x)y_{\leq l}^{\alpha}C < 1$ on $A$.  If $\widehat{a}(x)y_{\leq l}^{\alpha}C > 1$ on $A$, then $a_d$ is bounded below by a positive constant, and we are in Case 2.  So assume that $\widehat{a}(x)y_{\leq l}^{\alpha}C < 1$ on $A$.  Consider the two sets
\[
\{(x,y)\in A : a_d(x,y_{<d}) < y_d < \widehat{a}(x)y_{\leq l}^{\alpha}C\}
\quad\text{and}\quad
\{(x,y)\in A : \widehat{a}(x)y_{\leq l}^{\alpha}C < y_d < 1\}.
\]
By restricting $\varphi$ to the first set and then pulling back by a blowup in $y_d$, we reduce to Case 2.  By restricting $\varphi$ to the second set and then swapping the coordinates $y_{l+1}$ and $y_d$, we reduce to the case that $\varphi_{\leq d}$ is $(l+1)$-rectilinear and $\varphi$ is prepared over $\RR^{m+d}$, and we are done.  This completes the proof when $\supp(\alpha_{>l})$ is empty.

Now suppose that $\supp(\alpha_{>l})$ is nonempty.  By pulling back by a swap, we may assume that $l+1\in\supp(\alpha_{>l})$.  By pulling back by the power substitution $y_d\mapsto y_{d}^{\alpha_{l+1}}$, we may also assume that $\alpha_{l+1} = 1$.  Let $y'$ and $\alpha'$ be the tuples indexed by $\{1,\ldots,d-1\}\setminus\{l+1\}$ that are respectively obtained from $y_{<d}$ and $\alpha$ by omitting their $(l+1)$-th components, and write $y_{<d} = (y',y_{l+1})$; thus $\alpha_{>l} = (1,\alpha_{>l+1})$ and $\alpha'_{>l} = \alpha_{>l+1}$.  Fix a constant $C > 1$ that is greater than the supremum of the range of $\widehat{a}(x)(y')^{\alpha'}u(x,y',y_{l+1})$; this may may done because $\widehat{a}(x) (y')^{\alpha'} y_{l+1}$ is bounded (since it equals $a_d(x,y_{<d})/u(x,y_{<d})$) and $y_{l+1}$ may freely approach $1$ independently of the other variables.  Thus
\[
a_d(x,y',y_{l+1}) = \widehat{a}(x)(y')^{\alpha'} y_{l+1} u(x,y',y_{l+1}) < C y_{l+1}
\]
on $A$.  Consider the three sets,
\[
\{(x,y)\in A : C^{-1} < y_{l+1} < 1\},
\]
\[
\{(x,y)\in A : \text{$0 < y_{l+1} < C^{-1}$ and $a(x,y',y_{l+1}) < y_d < C y_{l+1}$}\}
\]
and
\[
\{(x,y)\in A : \text{$0 < y_{l+1} < C^{-1}$ and $C y_{l+1} < y_d < 1$}\}.
\]
By restricting $\varphi$ to the first set, we reduce to the case that $\varphi_{\leq d}$ is $(l+1)$-rectilinear, and we are done by the induction hypothesis since $|\supp(\alpha_{>l+1})| < |\supp(\alpha_{>l})|$.  If we restrict $\varphi$ to either the second or third set, we may pull back by a blowup in $y_{l+1}$ to assume that $C = 1$.  On the second set, we may then pull back by a blowup in $y_d$, and we are done by the induction hypothesis since $|\supp(\alpha'_{>l})| < |\supp(\alpha_{>l})|$.  The third set can also be written as $\{(x,y)\in A : 0 < y_d < 1, 0 < y_{l+1} < y_d\}$, so we may reduce to Case 1 by swapping the coordinates $y_{l+1}$ and $y_d$.
\end{proof}

%% file: constrRect.tex
\section{Rectilinear Preparation of Constructible Functions}
\label{s:constrRect}

This section proves the following proposition, which is a preparation result for constructible functions in transformed coordinates on rectilinear sets.

\begin{proposition}\label{prop:constrRect}
Let $\F$ be a finite set of constructible functions on a subanalytic set $D\subset\RR^{m+n}$.
There exists an open partition $\A$ of $D$ over $\RR^m$ such that for each $A\in\A$ there exist a subanalytic analytic isomorphism $F=(F_1,\ldots,F_{m+n}):B\to A$ over $\RR^m$, rational monomial maps $\varphi$ on $A$ and $\psi$ on $B$ over $\RR^m$, and $l\in\{0,\ldots,n\}$ with the following properties.
\begin{enumerate}{\setlength{\itemsep}{3pt}
\item
Pullback Property:
The map $\psi$ is $l$-rectilinear over $\RR^m$, $\det\PD{}{F}{y}$ is $\psi$-prepared,
and for every $f\in\F$ we may write $f\circ F$ in the form
\begin{equation}\label{eq:constrRectMainSum}
f\circ F(x,y) = \sum_{s\in S} (\log y_{>l})^s \left(\sum_{r\in R_{s}^{\CR}} y_{>l}^{r} f_{r,s}(x,y_{\leq l}) + \sum_{r\in R^{\NC}_{s}} y_{>l}^{r} f_{r,s}(x,y) \right)
\end{equation}
on $B$, where the sets $S\subset\NN^{n-l}$ and $R_{s}^{\CR},R_{s}^{\NC}\subset\ZZ^{n-l}$ are finite with $R^{\CR}_{s}\cap R^{\NC}_{s} = \emptyset$ for each $s$, and each function $f_{r,s}$ may be written as a finite sum
\begin{equation}\label{eq:constrRectSubSum}
\left\{\begin{array}{rcl}
f_{r,s}(x,y_{\leq l})
    & = &
    \displaystyle
    \sum_{j} g_j(x) y_{\leq l}^{\alpha_j}(\log y_{\leq l})^{\beta_j} h_j(x,y_{\leq l}),
    \quad
    \text{if $r\in R^{\CR}_{s}$,}
\vspace*{10pt}
\\
f_{r,s}(x,y)
    & = &
    \displaystyle
    \sum_{j} g_j(x) y_{\leq l}^{\alpha_j}(\log y_{\leq l})^{\beta_j} h_j(x,y),
    \quad
    \text{if $r\in R^{\NC}_{s}$,}
\end{array}\right.
\end{equation}
where $g_j\in\C(\Pi_m(A))$, $\alpha_j\in\ZZ^l$, $\beta_j\in\NN^l$, $h_j$ is either a $\psi_{\leq l}$-function or a $\psi$-function according to whether $r$ is in $R^{\CR}_{s}$ or $R^{\NC}_{s}$, and the following holds:
\begin{equation}\label{eq:CR}
\left\{
\text{\parbox{5.5in}{
For each $s\in S$, $r'\in R^{\NC}_{s}$ and $(x,y_{\leq l})\in\Pi_{m+l}(B)$, if $f_{r',s}(x,y_{\leq l}, y_{>l})\neq 0$ for some $y_{>l}\in(0,1)^{n-l}$, then $f_{r,s}(x,y_{\leq l})\neq 0$ for some $r\in R^{\CR}_{s}$ with $r\leq r'$.
}}
\right.
\end{equation}

\item
Pushforward Property:
The components of $F^{-1}$ are $\varphi$-prepared, and $\psi\circ F^{-1}$ is a $\varphi$-function.
}\end{enumerate}
\end{proposition}

The superscripts ``$\CR$'' and ``$\NC$'' in the notation $R^{\CR}_{s}$ and $R^{\NC}_{s}$ stand for {\bf\emph{critical}} and {\bf\emph{noncritical}}.  In Section \ref{s:mainThmProofs} we will use \eqref{eq:CR} to see that the $L^p$-classes of $f(x,\cdot)$ are determined by which of the terms $f_{r,s}(x,\cdot)$ with $r\in R_{s}^{\CR}$ are identically zero, so in this sense these are the ``critical'' terms.

In the degenerate case of $l = n$, \eqref{eq:constrRectMainSum} and \eqref{eq:constrRectSubSum} simply mean that
\[
f\circ F(x,y) = \sum_{j} g_j(x) y^{\alpha_j}(\log y)^{\beta_j} h_j(x,y)
\]
for some constructible functions $g_j$, tuples $\alpha_j\in\ZZ^n$ and $\beta_j\in\NN^n$, and $\psi$-functions $h_j$.  To see this, note that if $f\circ F$ is nonzero and $l=n$, then $S = \NN^0 = \{0\}$ and $R_{0}^{\CR},R_{0}^{\NC}\subset\ZZ^0 = \{0\}$ with $R_{0}^{\CR}\cap R_{0}^{\NC} = \emptyset$, so $R_{0}^{\CR} = \{0\}$ and $R_{0}^{\NC} = \emptyset$ by \eqref{eq:CR}.

\begin{notation}
For any set $E\subset\RR^m$, let $\O_E$ denote the ring of all analytic germs on $E$, and let $\O_E[y]$ denote the ring of all polynomials in $y = (y_1,\ldots,y_n)$ with coefficients in $\O_E$.  Each member of $\O_E[y]$ is an equivalence class of functions defined on neighborhoods of $E\times\RR^n$ in $\RR^{m+n}$, and hence defines a function on $E\times\RR^n$.  For each $\F\subset\O_E[y]$, define the variety of $\F$ by
\[
\VV(\F) = \{(x,y)\in E\times\RR^n : \text{$f(x,y) = 0$ for all $f\in\F$}\}.
\]
\end{notation}

For each $x\in\RR^m$, the ring $\O_{\{x\}}$ is Noetherian, so $\O_{\{x\}}[y]$ is as well.  This implies that when $E$ is compact, the varieties of $\O_E[y]$ form the collection of closed subsets of a Noetherian topological space on $E\times\RR^n$; in other words, for any $\F\subset\O_E[y]$ there exists a finite $\F'\subset\F$ such that $\VV(\F') = \VV(\F)$.

\begin{notation}
We partially order $\NN^k$ by defining $\alpha \leq \beta$ if and only if $\alpha_j\leq \beta_j$ for all $j\in\{1,\ldots,k\}$, where $\alpha = (\alpha_1,\ldots,\alpha_k)$ and $\beta = (\beta_1,\ldots,\beta_k)$.  For any $\alpha\in\NN^k$ write $[\alpha] = \{\beta\in\NN^k : \beta\geq\alpha\}$, and for any $A\subset\NN^k$ write $[A] = \bigcup_{\alpha\in A}[\alpha]$ for the upward closure of $A$.  If $A\subset\NN^k$ is nonempty, define $\min A$ to be the set of minimal members of $A$, and define $\min\emptyset = \emptyset$.
\end{notation}

Dickson's lemma states that $\min A$ is finite for every $A\subset\NN^k$.  The following is a parameterized version of Dickson's lemma.

\begin{lemma}\label{lemma:Dickson}
Let $E\subset\RR^m$ be compact and $\{f_\alpha\}_{\alpha\in\NN^k} \subset\O_E[y]$.  Then the set
\begin{equation}\label{eq:Dickson}
\bigcup_{(x,y)\in E\times\RR^n} \min\{\alpha\in\NN^k : f_\alpha(x,y)\neq 0\}
\end{equation}
is finite.
\end{lemma}

\begin{proof}
The proof is by induction on $k$, with the base case of $k = 0$ being trivial.  For the inductive step, use topological Noetherianity to fix $\beta\in\NN^k$ such that $\VV(\{f_\alpha\}_{\alpha\leq\beta}) = \VV(\{f_\alpha\}_{\alpha\in\NN^k})$.  Then \eqref{eq:Dickson} is finite because it is contained in
\begin{equation}\label{eq:Dicksonproof}
\bigcup_{i=1}^{k}\bigcup_{j=0}^{\beta_i} \left(\bigcup_{(x,y)\in E\times\RR^n} \min\{\alpha\in\NN^n : f_\alpha(x,y)\neq 0, \alpha_i = j\}\right),
\end{equation}
and each of the sets in parenthesis in \eqref{eq:Dicksonproof} is finite by the induction hypothesis.
\end{proof}

\begin{lemma}\label{lemma:partSupp}
Let $M\subset\NN^k$ be finite.  Then there exists a finite partition of $[M]\setminus M$ that is compatible with  $\{[\alpha]\}_{\alpha\in M}$ and is such that each member of the partition has a unique minimal member.
\end{lemma}

\begin{proof}
Define $\epsilon = (\epsilon_1,\ldots,\epsilon_k)$ by $\epsilon_i = \max\{\alpha_i : \alpha\in M\}$ for each $i\in\{1,\ldots,k\}$.  Let the partition of $[M]\setminus M$ consist of all the singletons $\{\alpha\}$ with $\alpha\in \left(\prod_{i=1}^{k}[0,\epsilon_i]\right)\cap[M]\setminus M$ and all sets of the form
\[
\left\{\alpha\in\NN^k : \left(\bigwedge_{i\in N} \alpha_i > \epsilon_i\right)\wedge \left(\bigwedge_{j\in\{1,\ldots,k\}\setminus N} \alpha_j = \beta_j\right)\right\},
\]
for each nonempty $N\subset\{1,\ldots,k\}$ and $\beta = (\beta_1,\ldots,\beta_k)$ in $\left(\prod_{i=1}^{k}[0,\epsilon_i]\right)\cap[M]$
\end{proof}

\begin{lemma}\label{lemma:analSum}
Let $E\subset\RR^m$ be compact, and suppose that $f$ is represented by a convergent power series
\[
f(x,y,z) = \sum_{\alpha\in\NN^k} f_\alpha(x,y)z^\alpha
\]
on $E\times\RR^n\times[0,1]^k$, where $f_\alpha\in\O_E[y]$ for each $\alpha\in\NN^k$.  Then we may write
\begin{equation}\label{eq:analSum}
f(x,y,z) = \sum_{\alpha\in M^{\CR}} z^\alpha f_\alpha(x,y) + \sum_{\beta\in M^{\NC}} z^\beta f_{\beta}(x,y,z)
\end{equation}
on $E\times\RR^n\times[0,1]^k$, where the sets $M^{\CR},M^{\NC}\subset\NN^k$ are finite and disjoint, each $f_\beta$ with $\beta\in M^{\CR}$ is represented by a subseries of $\sum_{\alpha\geq\beta} f_\alpha(x,y) z^{\alpha-\beta}$, and for each $(x,y)\in E\times\RR^n$ and each $\beta\in M^{\NC}$, if $f_\beta(x,y,z) \neq 0$ for some $z\in[0,1]^k$, then $f_\alpha(x,y)\neq 0$ for some $\alpha\in M^{\CR}$ with $\alpha\leq \beta$.
\end{lemma}

\begin{proof}
Let $M^{\CR}$ be the set defined in \eqref{eq:Dickson}, let $\S$ be the partition of $[M^{\CR}]\setminus M^{\CR}$ given by Lemma \ref{lemma:partSupp}, and let $M^{\NC}$ be the set of minimal members of the sets in $\S$.  For each $\beta\in M^{\NC}$, write $S_\beta$ for the unique member of $\S$ whose minimal member is $\beta$, and define $f_{\beta}(x,y,z) = \sum_{\alpha\in S_\beta} f_\alpha(x,y)z^{\alpha-\beta}$.  Then \eqref{eq:analSum} holds.  Consider $\beta\in M^{\NC}$ and $(x,y)\in E\times\RR^n$ such that $f_{\beta}(x,y,z)\neq 0$ for some $z\in[0,1]^k$.  Then $f_\gamma(x,y)\neq 0$ for some $\gamma\in S_\beta$.  Fix $\alpha\in M^{\CR}$ such that $f_\alpha(x,y)\neq 0$ and $\alpha\leq\gamma$.  Thus $S_\beta\cap[\alpha]$ is nonempty, so $S_\beta\subset[\alpha]$ by the compatibility property of $\S$, and hence $\alpha\leq \beta$.
\end{proof}

\begin{proof}[Proof of Proposition \ref{prop:constrRect}]
For each $f\in\F$ write $f(x,y) = \sum_i f_i(x,y) \prod_j \log f_{i,j}(x,y)$ for finitely many subanalytic functions $f_i:D\to\RR$ and $f_{i,j}:D\to(0,\infty)$.  Apply Proposition \ref{prop:subRect} to $\bigcup_{f\in\F}\{f_i,f_{i,j}\}_{i,j}$, and focus on one set $A$ in the open partition of $D$ over $\RR^m$ that this gives, along with its associated maps $F:B\to A$, $\varphi$ on $A$, and $\psi$ on $B$, where $\psi$ is $l$-rectilinear over $\RR^m$.  Thus $\det\PD{}{F}{y}$ is $\psi$-prepared, and we may write
\[
f\circ F(x,y) = \sum_i a_{i}(x) y^{\alpha_i} u_i(x,y) \prod_j \log a_{i,j}(x) y^{\alpha_{i,j}} u_{i,j}(x,y)
\]
on $B$ for some analytic subanalytic functions $a_i$ and $a_{i,j}$, tuples $\alpha_i$ and $\alpha_{i,j}$ in $\QQ^n$, and $\psi$-units $u_i$ and $u_{i,j}$.  By expanding the logarithms and distributing, we may rewrite this in the form
\begin{equation}\label{eq:constrRect1}
f\circ F(x,y) = \sum_i g_i(x) y^{\alpha_i} (\log y)^{\beta_i} h_i(x,y)
\end{equation}
for some constructible functions $g_i$, tuples $\alpha_i\in\QQ^n$ and $\beta_i\in\NN^n$, and $\psi$-functions $h_i$.  By pulling back by power substitutions in $y$, we may assume that $\alpha_i\in\ZZ^n$ for each $\alpha_i$ in \eqref{eq:constrRect1}.  Write $h_i(x,y) = H_i(\psi_{\leq l}(x,y_{\leq l}), y_{>l})$ for some analytic function $H_i(X,y_{>l})$ on the closure of the image of $\psi$.

We are done if $l = n$, so assume that $l < n$ and work by induction on $n-l$.  Since the closure of the range of $\psi_{\leq l}$ is compact, we may fix $\epsilon > 0$ such that each function $H_i$ is given by a single convergent power series in $y_{>l}$ with analytic coefficients in $(X,y_{\leq l})$, say
\begin{equation}\label{eq:HpowerSeries}
H_i(X,y_{>l}) = \sum_{\gamma\in\NN^{n-l}} H_{i,\gamma}(X) y_{>l}^{\gamma},
\end{equation}
for all $X$ in the closure of the range of $\psi_{\leq l}$ and all $y_{>l}$ in $[0,\epsilon]^{n-l}$.  For each $j\in\{l+1,\ldots,n\}$, by restricting $\psi$ to $\{(x,y)\in B : y_j > \epsilon\}$ and swapping the coordinates $y_{l+1}$ and $y_j$, we may reduce to the case that $\psi$ is $(l+1)$-rectilinear, in which case we are done by our induction on $n-l$.  So it suffices to restrict $\psi$ to $B\cap(\RR^{m+l}\times(0,\epsilon)^{n-l})$.  After pulling back by the maps sending $y_j\mapsto \epsilon y_j$ for each $j\in\{l+1,\ldots,n\}$, and again expanding the logarithms $\log y_j\epsilon = \log y_j + \log \epsilon$ and distributing, we may assume that $\epsilon = 1$.  We are now done pulling back $\psi$.  The pushforward property of the proposition we are proving follows from the fact that $\varphi$ satisfies the pushforward property of Proposition \ref{prop:subRect}, because we have only applied some very simple pullback constructions to the map $\psi$ originally given by Proposition \ref{prop:subRect}.  It remains to show that we can express $f\circ F$ as a sum in the desired form.

By grouping terms in \eqref{eq:constrRect1} according to like powers of $\log y_{>l}$, factoring out suitable monomials in $y$, and absorbing any remaining monomials in $y_{>l}$ with nonnegative powers inside of $\psi$-functions, we may rewrite \eqref{eq:constrRect1} in the form
\begin{equation}\label{eq:constrRect2}
f\circ F(x,y) = \sum_{s\in S} (\log y_{>l})^s y^{\delta_s}
\sum_{j\in J_s} g_j(x) y_{\leq l}^{\alpha_j}(\log y_{\leq l})^{\beta_j} h_j(x,y)
\end{equation}
for some finite $S\subset\NN^{n-l}$ and finite index sets $J_s$, constructible functions $g_j$, tuples $\delta_s\in\ZZ^n$ and $\alpha_j,\beta_j\in\NN^l$, and $\psi$-functions $h_j$, which we still write as $h_j = H_j\circ\psi$ with $H_j$ written as a power series \eqref{eq:HpowerSeries}.  For each $s\in S$ write
\[
G_s\circ\Psi_s(x,y) = \sum_{j\in J_s} g_j(x) y_{\leq l}^{\alpha_j}(\log y_{\leq l})^{\beta_j} h_j(x,y),
\]
where
\begin{eqnarray*}
\Psi_s(x,y)
    & = &
    \left(
    \psi_{\leq l}(x,y_{\leq l}), \log y_{\leq l}, (g_j(x))_{j\in J_s}, y
    \right),
    \\
G_s(X,Y,Z_s,y)
    & = &
    \sum_{j\in J_s} Z_j y_{\leq l}^{\alpha_j} Y^{\beta_j} H_j(X,y_{>l}),
\end{eqnarray*}
with $Z_s = (Z_j)_{j\in J_s}$ and $Y = (Y_1,\ldots,Y_l)$.  By computing
\begin{equation}\label{eq:reverseSum}
\sum_{j\in J_s}\left( Z_j y_{\leq l}^{\alpha_j} Y^{\beta_j} \sum_{\gamma\in\NN^{n-l}} H_{j,\gamma}(X)y_{>l}^{\gamma} \right)
=
\sum_{\gamma\in\NN^{n-l}}\left(\sum_{j\in J_s} Z_j y_{\leq l}^{\alpha_j} Y^{\beta_j} H_{j,\gamma}(X)\right) y_{>l}^{\gamma},
\end{equation}
we may write
\[
G_s(X,Y,Z_s,y_{>l}) = \sum_{\gamma\in\NN^{n-l}} G_{s,\gamma}(X,Y,Z_s,y_{\leq l}) y_{>l}^{\gamma}
\]
with each
\[
G_{s,\gamma}(X,Y,Z_s,y_{\leq l}) = \sum_{j\in J_s} Z_j y_{\leq l}^{\alpha_j} Y^{\beta_j} H_{j,\gamma}(X).
\]

Note that each $G_{s,\gamma}$ is a polynomial in $(Y,Z_s,y_{\leq l})$ with analytic coefficients in $X$, and $X$ ranges over a compact set.   So we may apply Lemma \ref{lemma:analSum} to get
\[
G_s(X,Y,Z_s,y)
=
\sum_{\gamma\in R^{\CR}_{s}} y_{>l}^{\gamma} G_{s,\gamma}(X,Y,Z_s,y_{\leq l})
+
\sum_{\gamma\in R^{\NC}_{s}} y_{>l}^{\gamma} G_{s,\gamma}^{\NC}(X,Y,Z_s,y),
\]
where $R^{\CR}_{s}$ and $R^{\NC}_{s}$ are disjoint subsets of $\NN^{n-l}$, each $G_{s,\gamma}^{\NC}$ is an analytic function represented by a subseries of $\sum_{\delta\geq \gamma} y_{>l}^{\delta-\gamma} G_{s,\delta}(X,Y,Z_s,y_{\leq l})$, and for each choice of $(X,Y,Z_s,y_{\leq l})$ and $\gamma'\in R^{\NC}_{s}$, if $G^{\NC}_{s,\gamma'}(X,Y,Z_s,y_{\leq l}, y_{>l}) \neq 0$ for some $y_{>l}\in[0,1]^{n-l}$, then there exists $\gamma\in R^{\CR}_{s}$ such that $G_{s,\gamma}(X,Y,Z_s,y_{\leq l})\neq 0$ and $\gamma\leq \gamma'$.  Write
\begin{equation}\label{eq:constrRect3}
f\circ F(x,y) = \sum_{s\in S}(\log y_{>l})^s y^{\delta_s}
\left(
\sum_{\gamma\in R^{\CR}_{s}} y_{>l}^{\gamma} G_{s,\gamma}\circ\Psi_{s,\leq l}(x,y_{\leq l})
+
\sum_{\gamma\in R^{\NC}_{s}} y_{>l}^{\gamma} G_{s,\gamma}^{\NC}\circ\Psi_s(x,y)
\right),
\end{equation}
where $\Psi_{s,\leq l}$ is the map obtained from $\Psi_s$ by omitting its components $y_{>l}$.  By distributing each $y^{\delta_s}$ and expressing each function $G_{s,\gamma}^{\NC}$ as a sum of terms indexed by $j\in J_s$, via a computation analogous to what was done in \eqref{eq:reverseSum} for $G_s$ (but going from right to left rather than from left to right), we see that \eqref{eq:constrRect3} expresses $f\circ F$ in the desired form.
\end{proof}

%% file: mainThmProofs.tex
\section{Proofs of the Main Theorems}\label{s:mainThmProofs}

This section proves Theorem \ref{thm:LC} in the constructible case, and it states and proves our main preparation theorem for constructible functions, from which Theorem \ref{thm:constrPrepSimple} follows as a special case.  We begin by fixing some notation to describe a situation that will be encountered throughout the section.

\begin{notation}\label{notation:mainThmSetup1}
Consider a finite set $\F$ of constructible functions on a subanalytic set $D\subset\RR^{m+n}$, and let $\A$ be an open partition of $D$ over $\RR^m$ obtained by applying Proposition \ref{prop:constrRect} to $\F$.  Focus on one $A\in\A$, along with its associated maps $F = (F_1,\ldots,F_{m+n}):B\to A$, $\varphi$ on $A$, and $\psi$ on $B$, where $\psi$ is $l$-rectilinear over $\RR^m$, as in the statement of the  proposition.  Write $(x,\tld{y})$ for the coordinates on $A$ with center $\theta$, where $\theta$ is the center of $\varphi$.  Write
\[
\det \PD{}{F}{y}(x,y) = H(x) y^\gamma U(x,y)
\]
on $B$ for some analytic subanalytic function $H$, tuple $\gamma = (\gamma_1,\ldots,\gamma_n)$ in $\QQ^n$, and $\varphi$-unit $U$.  For each $f\in\F$ write equation \eqref{eq:constrRectMainSum} as
\[
f\circ F(x,y) = \sum_{(r,s)\in\Delta(f,A)} \widehat{f}_{r,s}(x,y)
\]
on $B$, where
\begin{eqnarray*}
\Delta^{\CR}(f,A)
    & = &
    \{(r,s) : \text{$s\in S$ and  $r\in R_{s}^{\CR}$}\},
    \\
\Delta^{\NC}(f,A)
    & = &
    \{(r,s) : \text{$s\in S$ and  $r\in R_{s}^{\NC}$}\},
    \\
\Delta(f,A)
    & = &
    \Delta^{\CR}(f,A) \cup \Delta^{\NC}(f,A),
    \\
\widehat{f}_{r,s}(x,y)
    & = &
    \begin{cases}
    y_{>l}^{r}(\log y_{>l})^s f_{r,s}(x,y_{\leq l}),
        & \text{if $(r,s)\in\Delta^{\CR}(f,A)$,} \\
    y_{>l}^{r}(\log y_{>l})^s f_{r,s}(x,y),
        & \text{if $(r,s)\in\Delta^{\NC}(f,A)$,}
    \end{cases}
\end{eqnarray*}
for the sets $S$, $R_{s}^{\CR}$ and $R_{s}^{\NC}$ and the functions $f_{r,s}$ defined from $f$ and $A$ in Proposition \ref{prop:constrRect}.
For each $f\in\F$ and $x\in\Pi_m(A)$, define
\begin{eqnarray*}
\Delta^{\CR}(f,A,x)
    & = &
    \{(r,s)\in\Delta^{\CR}(f,A) : \text{$f_{r,s}(x,y_{\leq l}) \neq 0$ for some $y_{\leq l}\in\Pi_l(B_x)$}\},
    \\
\Delta^{\NC}(f,A,x)
    & = &
    \{(r,s)\in\Delta^{\NC}(f,A) : \text{$f_{r,s}(x,y) \neq 0$ for some $y\in B_x$}\},
    \\
\Delta(f,A,x)
    & = &
    \Delta^{\CR}(f,A,x) \cup \Delta^{\NC}(f,A,x),
    \\
\Omega(f,A,x)
    & = &
    \{y_{\leq l} \in \Pi_l(B_x) : \text{$f_{(r,s)}(x,y_{\leq l}) \neq 0$ for all $(r,s)\in \Delta^{\CR}(f,A,x)$}\}.
\end{eqnarray*}
For each $x\in\Pi_m(A)$ and $i\in\{l+1,\ldots,n\}$, define
\begin{eqnarray*}
\bar{r}_i(f,A,x)
    & = &
    \inf\{r_i : (r,s)\in\Delta^{\CR}(f,A,x)\},
    \\
\bar{s}_i(f,A,x)
    & = &
    \sup\{s_i : \text{$(r,s)\in\Delta^{\CR}(f,x)$ and $r_i = \bar{r}_i(f,A,x)$}\},
\end{eqnarray*}
under the convention that $\bar{r}_i(f,A,x) = \infty$ and $\bar{s}_i(f,A,x) = 0$ when $\Delta^{\CR}(f,A,x)$ is empty.
\end{notation}

\begin{remarks}\label{rmk:mainThmSetup1}
Consider the situation described in Notation \ref{notation:mainThmSetup1}, and let $f\in\F$.
\begin{enumerate}{\setlength{\itemsep}{3pt}
\item
For each $x\in\Pi_m(A)$, the set $\Omega(f,A,x)$ is dense and open in $\Pi_l(B_x)$.

\begin{proof}
This follows from the fact that for each $x\in\Pi_m(A)$ and $(r,s)\in\Delta^{\CR}(f,A,x)$, $f_{r,s}(x,\cdot)$ is a nonzero analytic function on $\Pi_l(B_x)$, and $\Pi_l(B_x)$ is connected and open in $\RR^l$.
\end{proof}

\item
For each $x\in\Pi_m(A)$, the set $\Delta^{\CR}(f,A,x)$ is empty if and only if $f(x,y) = 0$ for all $y\in A_x$.

\begin{proof}
If $\Delta^{\CR}(f,A,x)$ is empty, then \eqref{eq:CR} implies that $f(x,\cdot)$ is identically zero on $A_x$.  If $\Delta^{\CR}(f,A,x)$ is nonempty, then the following lemma implies that $f(x,\cdot)$ is not identically zero on $A_x$.
\end{proof}
}\end{enumerate}
\end{remarks}

\begin{lemma}\label{lemma:asympEst}
Consider the situation described in Notation \ref{notation:mainThmSetup1}.  Fix $f\in\F$, $i\in\{l+1,\ldots,n\}$, $x\in\Pi_m(A)$ with $\Delta^{\CR}(f,A,x)\neq \emptyset$, and $y_{\leq l}\in \Omega(f,A,x)$.  For any tuple $y_{>l} = (y_{l+1},\ldots,y_n)$, write $y' = (y_j)_{j\in\{l+1,\ldots,n\}\setminus\{i\}}$ and $y_{>l} = (y',y_i)$.  Then the limit
\begin{equation}\label{eq:asymEstLimit}
\lim_{y_i\to 0} \frac{f\circ F(x,y)}{y_{i}^{\bar{r}_i(f,A,x)} (\log y_i)^{\bar{s}_i(f,A,x)}}
\end{equation}
exists for all $y'\in(0,1)^{n-l-1}$, and the set
\begin{equation}\label{eq:asympEstSet}
\left\{y'\in(0,1)^{n-l-1} : \text{\eqref{eq:asymEstLimit} is nonzero}\right\}
\end{equation}
is dense and open in $(0,1)^{n-l-1}$.
\end{lemma}

\begin{proof}
Define
\begin{eqnarray*}
\Delta_i(f,A,x)
    & = &
    \{(r,s)\in\Delta(f,A,x) : \text{$r_i = \bar{r}_i(f,x)$ and $s_i = \bar{s}_i(f,x)$}\},
    \\
\Delta^{\CR}_{i}(f,A,x)
    & = &
    \Delta_i(f,A,x) \cap \Delta^{\CR}(f,A),
    \\
\Delta^{\NC}_{i}(f,A,x)
    & = &
    \Delta_i(f,A,x) \cap \Delta^{\NC}(f,A).
\end{eqnarray*}
It follows from \eqref{eq:CR} that for each $(r,s)\in\Delta(f,A,x)$, either $r_i > \bar{r}_i(f,A,x)$, or $r_i = \bar{r}_i(f,A,x)$ and $s_i\leq \bar{s}_i(f,A,x)$.  Therefore the limit \eqref{eq:asymEstLimit} exists and equals $g(y')$, where $g:(0,1)^{n-l-1}\to\RR$ is the analytic function defined by
\[
g(y') = \sum_{(r,s)\in\Delta^{\CR}_{i}(f,A,x)} (y')^{r'} (\log y')^{s'} f_{r,s}(x,y_{\leq l})
\,\,+\!\!
\sum_{(r,s)\in\Delta^{\NC}_{i}(f,A,x)} (y')^{r'} (\log y')^{s'} f_{r,s}(x,y_{\leq l}, y',0).
\]
So to prove that \eqref{eq:asympEstSet} is dense and open in $(0,1)^{n-l-1}$, it suffices to show that $g$ is not identically zero.  To do that we will show that $g\circ\eta$ is not identically zero, where $\eta:\Lambda\times(0,1)\to(0,1)^{n-l-1}$ is defined by
\[
\eta(\lambda,t) = (t^{\lambda_j})_{j\in\{l+1,\ldots,n\}\setminus\{i\}}
\]
for some suitably chosen open set $\Lambda\subset(0,\infty)^{n-l-1}$.

Note that
\begin{eqnarray*}
g\circ\eta(\lambda,t)
    & = &
    \sum_{(r,s)\in\Delta^{\CR}_{i}(f,A,x)} t^{\lambda\cdot r'} \lambda^{s'}(\log t)^{|s'|} f_{r,s}(x,y_{\leq l})
    \\
    & &
    +
    \sum_{(r,s)\in\Delta^{\NC}_{i}(f,A,x)} t^{\lambda\cdot r'} \lambda^{s'}(\log t)^{|s'|} f_{r,s}(x,y_{\leq l},\eta(\lambda,t),0).
\end{eqnarray*}
We may choose $\Lambda$ so that there exist $\bar{r}' \in \{r' : (r,s)\in\Delta^{\CR}_{i}(f,A,x)\}$ and $c > 0$ such that for all $(r,s)\in\Delta^{\CR}_{i}(f,A,x)$ with $r'\neq \bar{r}'$,
\begin{equation}\label{eq:r'}
\lambda\cdot \bar{r}' + c < \lambda\cdot r'
\quad\text{for all $\lambda\in\Lambda$.}
\end{equation}
By \eqref{eq:CR}, for each $(r,s)\in\Delta^{\NC}_{i}(f,A,x)$ there exists $\rho$ such that $(\rho,s)\in\Delta^{\CR}_{i}(f,A,x)$ and $\rho\leq r$ (and necessarily $\rho\neq r$), so $\lambda\cdot \rho' < \lambda\cdot r'$ for all $\lambda\in\Lambda$.  Therefore by shrinking $\Lambda$ and $c$, we can ensure that \eqref{eq:r'} also holds for all $(r,s)\in\Delta^{\NC}_{i}(f,A,x)$.  So by defining
\begin{eqnarray*}
\bar{s}'
    & = &
    \max\{|s'| : \text{$(r,s)\in\Delta^{\CR}_{i}(f,A,x)$ and $r' = \bar{r}'$}\},
    \\
\Delta^{\CR}_{i,\Lambda}(f,A,x)
    & = &
    \{(r,s)\in\Delta^{\CR}_{i}(f,A,x) : \text{$r' = \bar{r}'$ and $|s'| = \bar{s}'$}\},
\end{eqnarray*}
we see that as $t$ tends to $0$, $g\circ\eta(\lambda,t)$ is asymptotic with
\[
t^{\lambda\cdot\bar{r}'} (\log t)^{\bar{s}'}
\left(
\sum_{(r,s)\in\Delta^{\CR}_{i,\Lambda}(f,A,x)} \lambda^{s'} f_{r,s}(x,y_{\leq l})
\right),
\]
which is not identically zero because the sum in parentheses is a nonzero polynomial in $\lambda$.
\end{proof}

To prove the next lemma, we need the following inequality:
\begin{equation}\label{eq:weakTriangle}
\text{$(x_1+\cdots+x_k)^p \leq x_{1}^{p} +\cdots + x_{k}^{p}$ if $x_1,\ldots,x_k\geq 0$ and $0<p\leq 1$.}
\end{equation}
The inequality \eqref{eq:weakTriangle} can be verified when $k=2$ by considering $f(t) = (x_1+t)^p$ and $g(t) = x_{1}^{p} + t^p$, where $x_1\geq 0$ and $0 < p \leq 1$, and then showing that $f(0) = g(0)$ and $f'(t) \leq g'(t)$ for all $t > 0$.  The general case then follows by induction on $k$.

\begin{lemma}\label{lemma:triangle}
Let $\nu$ be a positive measure on a set $Y$, let $\{f_i\}_{i\in I}$ and $\{g_j\}_{j\in J}$ be finite families of real-valued $\nu$-measurable functions on $Y$, and let $p,q>0$.  Put $M = \max\{p,q\}$. Then
\[
\int_Y \left|\sum_{i\in I} f_i\right|^p \left|\sum_{j\in J} g_j\right|^q d\nu
\leq
\begin{cases}
\displaystyle
\sum_{(i,j)\in I\times J} \int_Y |f_i|^p|g_j|^q d\nu,
    & \text{if $M < 1$,}
    \vspace*{5pt}\\
\displaystyle
\left(\sum_{(i,j)\in I\times J} \left(\int_Y |f_i|^p|g_j|^q d\nu\right)^{1/M}\right)^M,
    & \text{if $M\geq 1$.}
\end{cases}
\]
\end{lemma}

\begin{proof}
By symmetry we may assume that $p\geq q$.  Then
\begin{eqnarray*}
\int_Y \left|\sum_{i\in I} f_i\right|^p \left|\sum_{j\in J} g_j\right|^q d\nu
    & \leq &
    \int_Y \left(\sum_{i\in I} |f_i|\right)^p \left(\sum_{j\in J} |g_j|\right)^q d\nu
    \\
    & = &
    \int_Y \left(\left(\sum_{i\in I}|f_i|\right)\left(\sum_{j\in J}|g_j|\right)^{q/p}\right)^p d\nu
    \\
    & \leq &
    \int_Y \left(\left(\sum_{i\in I}|f_i|\right)\left(\sum_{j\in J}|g_j|^{q/p}\right)\right)^p d\nu
    \quad\text{by \eqref{eq:weakTriangle},}
    \\
    & = &
    \int_Y \left(\sum_{(i,j)\in I\times J} |f_i||g_j|^{q/p}\right)^p d\nu,
    \\
    & \leq &
    \begin{cases}
    \displaystyle
    \sum_{(i,j)\in I\times J} \int_Y |f_i|^p |g_j|^q d\nu,
        & \text{if $p < 1$,}
    \vspace*{5pt}\\
    \displaystyle
    \left(\sum_{(i,j)\in I\times J} \left(\int_Y |f_i|^p |g_j|^q d\nu\right)^{1/p}\right)^p,
        & \text{if $p \geq 1$,}
    \end{cases}
\end{eqnarray*}
with the last inequality following from \eqref{eq:weakTriangle} when $p < 1$ and from the triangle inequality for $L^p(\nu)$ when $p\geq 1$.
\end{proof}

\begin{lemma}\label{lemma:OpenInt}
Consider the situation described in Notation \ref{notation:mainThmSetup1}, and suppose that $f,\mu\in \F$, $q > 0$ and $x\in\Pi_m(A)$.  Then
\[
\LC(f\Restr{A},|\mu|^q\Restr{A},x) \cap (0,\infty)
=
\bigcap_{i=l+1}^{n}\{p\in(0,\infty) : \bar{r}_i(f,A,x)p + \bar{r}_i(\mu,A,x)q + \gamma_i > -1\}.
\]
And, $\infty\in \LC(f\Restr{A},|\mu|^q\Restr{A},x)$ if and only if either $\Delta^{\CR}(\mu,A,x)$ is empty or else for each $i\in\{l+1,\ldots,n\}$, $\bar{r}_i(f,A,x) > 0$ or $\bar{r}_i(f,A,x) = \bar{s}_i(f,A,x) = 0$.
\end{lemma}

\begin{proof}
Let $x\in\Pi_m(A)$.  The conclusion is clear from Remark \ref{rmk:mainThmSetup1}.2 when either $\Delta^{\CR}(f,A,x)$ or $\Delta^{\CR}(\mu,A,x)$ is empty, for then $\LC(f,|\mu|^q,x) = (0,\infty]$ and either $\bar{r}_i(f,A,x) = \infty$ for all $i\in\{l+1,\ldots,n\}$ (when $\Delta^{\CR}(f,A,x)$ is empty), or $\bar{r}_i(\mu,A,x) = \infty$ for all $i\in\{l+1,\ldots,n\}$ (when $\Delta^{\CR}(\mu,A,x)$ is empty).  So we assume that $\Delta^{\CR}(f,A,x)$ and $\Delta^{\CR}(\mu,A,x)$ are both nonempty.  Let $p\in(0,\infty)$.

Suppose that
\begin{equation}\label{eq:linIneq}
\bar{r}_i(f,A,x)p + \bar{r}_i(\mu,A,x)q + \gamma_i > -1
\end{equation}
for all $i\in\{l+1,\ldots,n\}$.  Then
\[
r_i p + r'_i q + \gamma_i > -1
\]
for all $i\in\{l+1,\ldots,n\}$, $(r,s)\in\Delta(f,A,x)$ and $(r',s')\in\Delta(\mu,A,x)$.  By applying Lemma \ref{lemma:triangle} to the sums $f\circ F = \sum_{(r,s)} \widehat{f}_{r,s}$ and
$\mu\circ F = \sum_{(r,s)} \widehat{\mu}_{r,s}$ using the measure defined from the Jacobian of $F$ in $y$, and then by applying Corollary \ref{cor:rectIntBdd}, we see that $p\in\LC(f\Restr{A},|\mu|^q\Restr{A},x)$.

Conversely, suppose that $p\in\LC(f\Restr{A},|\mu|^q\Restr{A},x)$, and let $i\in\{l+1,\ldots,n\}$.  Fubini's theorem and Remark \ref{rmk:mainThmSetup1}.1 imply that there exist $y_{\leq l}\in \Omega(f,A,x)\cap \Omega(\mu,A,x)$ and $y'$ in the set \eqref{eq:asympEstSet} such that
\[
y_i \mapsto |f\circ F(x,y)|^p |\mu\circ F(x,y)|^q \det\PD{}{F}{y}(x,y)
\]
is integrable on $(0,1)$.  So \eqref{eq:linIneq} holds by Lemmas \ref{lemma:intBdd} and \ref{lemma:asympEst}.

The $L^\infty$ case is similar.  Indeed, suppose that $\bar{r}_i(f,A,x) > 0$ or $\bar{r}_i(f,A,x) = \bar{s}_i(f,A,x) = 0$ for all $i\in\{l+1,\ldots,n\}$.  Then $r_i > 0$ or $r_i = s_i= 0$ for all $i\in\{l+1,\ldots,n\}$ and $(r,s)\in\Delta(f,A,x)$.  So applying Corollary \ref{cor:rectIntBdd} to each term of the sum $f\circ F = \sum_{(r,s)} \widehat{f}_{r,s}$ shows that $f\circ F(x,\cdot)$ is bounded  on $B_x$, and hence $\infty\in\LC(f\Restr{A},|\mu|^q\Restr{A},x)$.

Conversely, suppose that $\infty\in\LC(f\Restr{A},|\mu|^q\Restr{A},x)$.  Then $f\circ F(x,\cdot)$ is bounded on $B_x$.  So for each $i\in\{l+1,\ldots,n\}$ we may choose $y_{\leq l}\in \Omega(f,A,x)$ and $y'$ in the set \eqref{eq:asympEstSet}, and thereby conclude that $\bar{r}_i(f,A,x) > 0$ or $\bar{r}_i(f,A,x) = \bar{s}_i(f,A,x) = 0$ by Lemmas \ref{lemma:intBdd} and \ref{lemma:asympEst}.
\end{proof}

\begin{proof}[Proof of Theorem \ref{thm:LC} in the Constructible Case]
Let $f,\mu\in\C(D)$ for a subanalytic set $D\subset\RR^{m+n}$, fix $q > 0$, and write $E = \Pi_m(D)$.  Apply Proposition \ref{prop:constrRect} to $\F = \{f,\mu\}$, and use Notation \ref{notation:mainThmSetup1}.  We claim that for each $A\in\A$, the set
\[
\I_A := \{\LC(f\Restr{A},|\mu|^q\Restr{A},x) : x\in\Pi_m(A)\}
\]
is a finite set of open subintervals of $(0,\infty]$ with endpoints in $(\Span_{\QQ}\{1,q\}\cap[0,\infty))\cup\{\infty\}$, and that for each $I\in\I_A$ there exists $g_{A,I}\in\C(\Pi_m(A))$ such that
\[
\left\{x\in\Pi_m(A) : I\subset\LC(f\Restr{A},|\mu|^q\Restr{A},x)\right\}
=
\{x\in\Pi_m(A) : g_{A,I}(x) = 0\}.
\]
The claim implies the theorem because for each $x\in E$,
\[
\LC(f,|\mu|^q,x) = \bigcap_{\scriptstyle A\in\A\,\text{s.t.} \atop \scriptstyle x\in\Pi_m(A)}
\LC(f\Restr{A},|\mu|^q\Restr{A},x),
\]
so the claim shows that $\I$ is a finite set of open subintervals of $(0,\infty]$ with endpoints in $(\Span_{\QQ}\{1,q\}\cap[0,\infty))\cup\{\infty\}$, and that for each $I\in\I$,
\begin{eqnarray*}
\{x\in E : I\subset\LC(f,|\mu|^q,x)\}
    & = &
    \{x\in E : \text{$I \subset \LC(f\Restr{A},|\mu|^q\Restr{A},x)$ for all $A\in\A$ with $x\in\Pi_m(A)$}\}
    \\
    & = &
    \left\{x\in E : \sum_{A\in\A}\sum_{\scriptstyle J\in\I_A\,\text{s.t.}\atop \scriptstyle I\subset J} (g'_{A,J}(x))^2 = 0\right\},
\end{eqnarray*}
where each $g'_{A,J}:E\to\RR$ is defined by extending $g_{A,J}$ by $0$ on $E\setminus\Pi_m(A)$.

To prove the claim, focus on one $A\in\A$.  Lemma \ref{lemma:OpenInt} shows that each member of $\I_A$ is an open subinterval of $(0,\infty]$ with endpoints in $(\Span_{\QQ}\{1,q\}\cap[0,\infty))\cup\{\infty\}$, and that $\I_A$ is finite because
\[
\LC(f\Restr{A},|\mu|^q\Restr{A},x) = \LC(f\Restr{A},|\mu|^q\Restr{A},x')
\]
for all $x,x'\in\Pi_m(A)$ such that $\Delta^{\CR}(f,A,x) = \Delta^{\CR}(f,A,x')$ and $\Delta^{\CR}(\mu,A,x) = \Delta^{\CR}(\mu,A,x')$.  Fix $I\in\I_A$.  We may define $g_{A,I} = 0$ if $I$ is empty, so assume that $I$ is nonempty.  Let $a = \inf I$ and $b = \sup I$.  Lemma \ref{lemma:OpenInt} implies that for any $x\in\Pi_m(A)$, when the infimum of $\LC(f\Restr{A},|\mu|^q\Restr{A},x)$ is finite, this infimum is determined by the inequalities \eqref{eq:linIneq} for all $i\in\{l+1,\ldots,n\}$ for which $\bar{r}_i(f,A,x)$ is positive; and similarly, when the supremum of $\LC(f\Restr{A},|\mu|^q\Restr{A},x)$ is finite, this supremum is determined by the inequalities \eqref{eq:linIneq} for all $i\in\{l+1,\ldots,n\}$ for which  $\bar{r}_i(f,A,x)$ is negative.  Therefore $I \subset \LC(f\Restr{A},|\mu|^q\Restr{A},x)$ if and only if each of the following two conditions hold.
\begin{enumerate}{\setlength{\itemsep}{3pt}
\item
If $I\cap(0,\infty)$ is nonempty, then
\[
\text{$f_{r,s}(x,y_{\leq l}) = 0$ and $\mu_{r',s'}(x,y_{\leq l}) = 0$ for all $y_{\leq l}\in\Pi_l(B_x)$,}
\]
for every $(r,s)\in\Delta^{\CR}(f,A)$ and $(r',s')\in\Delta^{\CR}(\mu,A)$ such that for all $i\in\{l+1,\ldots,n\}$,
\[
\begin{cases}
r_i a + r'_i q + \gamma_i < -1,
    & \text{if $r_i > 0$,}
    \\
r'_i q + \gamma_i \leq -1,
    & \text{if $r_i = 0$,}
    \\
r_i b + r'_i q + \gamma_i < -1,
    & \text{if $r_i < 0$,}
\end{cases}
\]
with the understanding that we are allowing computations in the extended real number system since $a$ or $b$ could be $\infty$.

\item
If $\infty\in I$, then at least one of the following two conditions hold.
\begin{enumerate}{\setlength{\itemsep}{3pt}
\item
We have
\[
\text{$\mu_{r',s'}(x,y_{\leq l}) = 0$ for all $y_{\leq l}\in\Pi_l(B_x)$,}
\]
for every $(r',s')\in\Delta^{\CR}(\mu,A)$.

\item
We have
\[
\text{$f_{r,s}(x,y_{\leq l}) = 0$ for all $y_{\leq l}\in\Pi_l(B_x)$,}
\]
for every $(r,s)\in\Delta^{\CR}(f,A)$ such that for all $i\in\{l+1,\ldots,n\}$, either $r_i < 0$, or else $r_i = 0$ and $s_i > 0$.
}\end{enumerate}
}\end{enumerate}
Therefore $g_{A,I}$ can be constructed using Theorem \ref{thm:vanish}.
\end{proof}

We now turn our attention to stating and proving the preparation theorem.

\begin{notation}\label{notation:mainThmSetup2}
When considering the situation described in Notation \ref{notation:mainThmSetup1}, we shall now also write $G = (G_1,\ldots,G_{m+n}):A\to B$ for the inverse of $F$, and for each $j\in\{l+1,\ldots,n\}$ write
\[
G_{m+j}(x,y) = H_{j}(x) |\tld{y}|^{\beta_j} V_j(x,y)
\]
on $A$, where $H_j$ is an analytic subanalytic function, $\beta_j\in\QQ^n$, and $V_j$ is a $\varphi$-unit.
\end{notation}

\begin{lemma}\label{lemma:Tk}
Consider the situation described in Notation \ref{notation:mainThmSetup1} and \ref{notation:mainThmSetup2}.  Let $f\in\F$ and $(r,s)\in\Delta(f,A)$, where $r = (r_{l+1},\ldots,r_n)$ and $s = (s_{l+1},\ldots,s_n)$.  We may express $\widehat{f}_{r,s}\circ G$ in the form
\begin{equation}\label{eq:TkSumLemma}
\widehat{f}_{r,s}\circ G(x,y) = \sum_{k\in K_{r,s}(f,A)} T_k(x,y)
\end{equation}
on $A$, where $K_{r,s}(f,A)$ is a finite index set and for each $k\in K_{r,s}(f,A)$,
\begin{equation}\label{eq:TkFormLemma}
T_k(x,y) = g_k(x)\, G_{>m}(x,y)^{R_k} \left(\prod_{j=1}^{n} \left( \log |\tld{y}|^{\beta_j} \right)^{S_{k,j}} \right) u_k(x,y)
\end{equation}
for some $g_k\in\C(\Pi_m(A))$, tuples $R_k = (R_{k,1},\ldots,R_{k,n})\in\QQ^n$ and $S_k = (S_{k,1},\ldots,S_{k,n})\in\NN^n$ satisfying $R_{k,j} = r_j$ and $S_{k,j} \leq s_j$ for all $j\in\{l+1,\ldots,n\}$, and $\varphi$-units $u_k$.
\end{lemma}

\begin{proof}
By \eqref{eq:constrRectSubSum} we may write $\widehat{f}_{r,s}(x,y)$ as a finite sum of terms of the form
\begin{equation}\label{eq:term1}
g(x) y^R (\log y)^S h(x,y)
\end{equation}
on $B$, where $g\in\C(\Pi_m(A))$, the tuples $R = (R_1,\ldots,R_n)\in\QQ^n$ and $S = (S_1,\ldots,S_n)\in\NN^n$ satisfy $R_j = r_j$ and $S_j = s_j$ for all $j\in\{l+1,\ldots,n\}$, and $h$ is a $\psi$-function.  Pulling back \eqref{eq:term1} by $G$ gives
\[
g(x) G_{>m}(x,y)^R (\log G_{>m}(x,y))^S h\circ G(x,y)
\]
on $A$.  In the above equation, by writing
\[
\log G_{m+j}(x,y) = \log H_j(x) + \log |\tld{y}|^{\beta_j} + \log V_j(x,y)
\]
for each $j\in\{1,\ldots,n\}$, and then distributing, we obtain the desired form given in \eqref{eq:TkSumLemma} and \eqref{eq:TkFormLemma}, except that each $u_k$ is only a $\varphi$-function, not necessarily a $\varphi$-unit.  But then by writing $u_k = (u_k - c) + c$ for some sufficiently large constant $c$ so that $u_k-c$ and $c$ are both units, and then separating each term in \eqref{eq:TkSumLemma} into two terms, we may further assume that each $u_k$ in \eqref{eq:TkSumLemma} is a $\varphi$-unit.
\end{proof}

\begin{lemma}\label{lemma:TkPullback}
Consider a single term $T_k$ given in \eqref{eq:TkFormLemma}.  We may express $T_k\circ F$ as a finite sum
\begin{equation}\label{eq:tXiSum}
T_k\circ F(x,y) = \sum_{\zeta} g_\zeta(x) y^{R_k} (\log y)^{S_\zeta} h_\zeta(x,y)
\end{equation}
on $B$ for some $g_\zeta\in\C(\Pi_m(A))$, tuples $S_\zeta = (S_{\zeta,1},\ldots,S_{\zeta,n})\in\NN^n$ satisfying $S_{\zeta,j} \leq S_{k,j}$ for each $j\in\{1,\ldots,n\}$, and bounded functions $h_\zeta$.
\end{lemma}

\begin{proof}
Since
\[
|\tld{y}|^{\beta_j} = \frac{G_{m+j}(x,y)}{H_j(x) V_j(x,y)}
\]
for each $j\in\{1,\ldots,n\}$, it follows from \eqref{eq:TkFormLemma} that
\[
T_k\circ F(x,y)
=
g_k(x) y^{R_k}
\left(
\prod_{j=1}^{n}
\left(
\log\frac{y_j}{H_j(x) V_j\circ F(x,y)}
\right)^{S_{k,j}}
\right)
u_k\circ F(x,y)
\]
on $B$.  In the above equation, write
\[
\log\frac{y_j}{H_j(x) V_j\circ F(x,y)}
=
\log y_j - \log H_j(x) - \log V_j\circ F(x,y)
\]
for each $j\in\{1,\ldots,n\}$, and then distribute.
\end{proof}

\begin{theorem}[Preparation of Constructible Functions - Full Version]
\label{thm:constrPrep}
Let $\Phi$ be a finite subset of $\C(D)\times\C(D)\times(0,\infty)$ for some subanalytic set $D\subset\RR^{m+n}$. For each $(f,\mu,q)\in\Phi$ let
\[
\I(f,\mu,q) = \{\LC(f,|\mu|^q,x) : x\in \Pi_m(D)\},
\]
and let $\F = \{f,\mu : (f,\mu,q)\in\Phi\}$.  Then there exists an open partition $\A$ of $D$ over $\RR^m$ into subanalytic cells over $\RR^m$ such that for each $A\in\A$ there exist a rational monomial map $\varphi$ on $A$ over $\RR^m$ and rational numbers $\beta_{i,j}$, where $i,j\in\{1,\ldots,n\}$, for which we may express each $f\in\F$ in the form
\begin{equation}\label{eq:TkSum}
f(x,y) = \sum_{k\in K(f,A)} T_k(x,y)
\end{equation}
on $A$, where $K(f,A)$ is a finite index set and for each $k\in K(f,A)$,
\begin{equation}\label{eq:TkForm}
T_k(x,y) = g_k(x) \left(\prod_{i=1}^{n} |\tld{y}_i|^{r_{k,i}} \left( \log \prod_{j=1}^{n} |\tld{y}_j|^{\beta_{i,j}}\right)^{s_{k,i}}\right) u_{k}(x,y)
\end{equation}
for some $g_k\in\C(\Pi_m(A))$, rational numbers $r_{k,i}$, natural numbers $s_{k,i}$, and $\varphi$-units $u_k$, where we are writing $(x,\tld{y})$ for the coordinates on $A$ with center $\theta$, with $\theta$ being the center for $\varphi$.  Moreover, for each $f\in\F$ and $A\in\A$ there exists a partition $\P(f,A)$ of $K(f,A)$ described as follows.

For each $A\in\A$, $(f,\mu,q)\in\Phi$, $K\in\P(f,A)$, $\Lambda\in\P(\mu,A)$, and $I\in\I(f,\mu,q)$, at least one of the following two statements holds:
\begin{enumerate}{\setlength{\itemsep}{3pt}
\item
for all $(\kappa,\lambda)\in K\times\Lambda$, we have $\Pi_m(A)\times I \subset \LC(T_\kappa,|T_\lambda|^q,\Pi_m(A))$;

\item
for all $x\in\Pi_m(A)$ such that $I\subset \LC(f,|\mu|^q,x)$, either $\sum_{\kappa\in K} T_\kappa(x,y) = 0$ for all $y\in A_x$ or $\sum_{\lambda\in \Lambda} T_\lambda(x,y)=0$ for all $y\in A_x$;
}\end{enumerate}
and if statement 2 does not hold, then
\begin{equation}\label{eq:T'k}
\Pi_m(A)\times(I\setminus\{\infty\}) \subset \LC(T'_\kappa, |T'_\lambda|^q, \Pi_m(A))
\end{equation}
for all $(\kappa,\lambda)\in K\times \Lambda$ and all functions $T'_\kappa$ and $T'_\lambda$ of the form
\begin{eqnarray*}
T'_\kappa(x,y)
    & = &
    \prod_{i=1}^{n} |\tld{y}_i|^{r_{\kappa,i}} \left(\log\prod_{j=1}^{n} |\tld{y}_i|^{\beta'_{\kappa,i,j}}\right)^{s'_{\kappa,i}},
\\
T'_\lambda(x,y)
    & = &
    \prod_{i=1}^{n} |\tld{y}_i|^{r_{\lambda,i}} \left(\log\prod_{j=1}^{n} |\tld{y}_i|^{\beta'_{\lambda,i,j}}\right)^{s'_{\lambda,i}},
\end{eqnarray*}
where the $\beta'_{\kappa,i,j}, \beta'_{\lambda,i,j} \in\QQ$ and $s'_{\kappa,i}, s'_{\lambda,i} \in \NN$ are arbitrary and the $r_{\kappa,i},r_{\lambda,i}$ are as in \eqref{eq:TkForm}.
\end{theorem}

\begin{proof}
Apply Proposition \ref{prop:constrRect} to $\F$.  Fix $A\in\A$, and use the notation found in Notation \ref{notation:mainThmSetup1} and \ref{notation:mainThmSetup2} and in Lemmas \ref{lemma:Tk} and \ref{lemma:TkPullback}.  Lemma \ref{lemma:Tk} shows that each $f\in\F$ may be written in the form given in \eqref{eq:TkSum} and \eqref{eq:TkForm}, where each $T_k$ is defined as in \eqref{eq:TkFormLemma} and
\[
K(f,A) = \bigcup_{(r,s)\in\Delta(f,A)} K_{r,s}(f,A).
\]
For each $f\in\F$, define
\[
\P(f,A) = \{K_{r,s}(f,A)\}_{(r,s)\in\Delta(f,A)}.
\]

Now also fix $(f,\mu,q)\in\Phi$, $K\in\P(f,A)$, $\Lambda\in\P(\mu,A)$ and $I\in\I(f,\mu,q)$.  Write $K = K_{r,s}(f,A)$ and $\Lambda = K_{r',s'}(\mu,A)$ for some $(r,s)\in\Delta(f,A)$ and $(r',s')\in\Delta(\mu,A)$.  We are done if statement 2 in the last sentence of the theorem holds, so assume otherwise.  Therefore we may fix $x_0\in\Pi_m(A)$ such that $I\subset\LC(f,|\mu|^q,x_0)$, $(r,s)\in\Delta(f,A,x_0)$ and $(r',s')\in\Delta(\mu,A,x_0)$.  Lemma \ref{lemma:OpenInt} gives the following.
\begin{equation}\label{eq:finiteI}
\left\{\text{\parbox{5in}{
For all $p\in I\cap(0,\infty)$ and all $i\in\{l+1,\ldots,n\}$,
\begin{center}
$\bar{r}_i(f,A,x_0)p + \bar{r}_i(\mu,A,x_0)q + \gamma_i > -1$.
\end{center}
}}
\right.
\end{equation}
\begin{equation}\label{eq:inftyI}
\left\{
\text{\parbox{5in}{
If $\infty\in I$, then for all $i\in\{l+1,\ldots,n\}$,
\begin{center}
$\bar{r}_i(f,A,x_0) > 0$ or  $\bar{r}_i(f,A,x_0) = \bar{s}_i(f,A,x_0) = 0$.
\end{center}
}}
\right.
\end{equation}

Let $\kappa\in K$ and $\lambda\in\Lambda$.  Write $T_\kappa$ and $T_\lambda$ as in \eqref{eq:TkFormLemma} with $k = \kappa$ and $k = \lambda$, respectively, and write
\begin{eqnarray}
\label{eq:TkappaSum}
T_\kappa\circ F(x,y)
    & = &
    \sum_{\zeta} g_\zeta(x) y^{R_\kappa} (\log y)^{S_\zeta} h_\zeta(x,y),
    \\
\label{eq:TlambdaSum}
T_\lambda\circ F(x,y)
    & = &
    \sum_{\eta} g_\eta(x) y^{R_\lambda} (\log y)^{S_\eta} h_\eta(x,y),
\end{eqnarray}
as in \eqref{eq:tXiSum}.  Note that for each $i\in\{l+1,\ldots,n\}$,
\begin{equation}\label{eq:R}
\text{$R_{\kappa,i} = r_i \geq \bar{r}_i(f,A,x_0)$ and $R_{\lambda,i} = r'_i \geq \bar{r}_i(\mu,A,x_0)$.}
\end{equation}
So \eqref{eq:finiteI} holds with $R_{\kappa,i}$ and $R_{\lambda,i}$ in place of $\bar{r}_i(f,A,x_0)$ and $\bar{r}_i(\mu,A,x_0)$, respectively.  Therefore by Corollary \ref{cor:rectIntBdd}, Lemma \ref{lemma:triangle},  \eqref{eq:TkappaSum} and \eqref{eq:TlambdaSum}, it follows that
\[
\Pi_m(A) \times (I\setminus\{\infty\}) \subset \LC(T_k,|T_\lambda|^q,\Pi_m(A)).
\]
Note that the proof of this fact depends only the values of $r$ and $r'$, being independent the values of $\beta_1,\ldots,\beta_n$, $s$ and $s'$, so \eqref{eq:T'k} follows.

Now suppose that $\infty\in I$.  Note that for each $\zeta$ and $i\in\{l+1,\ldots,n\}$, we have $S_{\zeta,i} \leq S_{\kappa,i} \leq s_i$.  Combining this with \eqref{eq:R} shows that for each $i\in\{l+1,\ldots,n\}$, either $R_{\kappa,i} > 0$ or else $R_{\kappa,i} = S_{\zeta,i} = 0$ for all $\zeta$.  Therefore Corollary \ref{cor:rectIntBdd} and \eqref{eq:TkappaSum} show that $T_\kappa\circ F(x,\cdot)$ is bounded on $B_x$ for each $x\in\Pi_m(A)$.  So $\Pi_m(A)\times\{\infty\} \subset \LC(T_k,|T_\lambda|^q,\Pi_m(A))$.

This completes the proof of the theorem, except for the fact that $A$ need not be a cell over $\RR^m$.  To remedy this, simply construct an open partition of $A$ over $\RR^m$ consisting of cells over $\RR^m$ (for instance, using Proposition \ref{prop:subPrep}), and then restrict to each of these cells.
\end{proof}

Theorem \ref{thm:constrPrep} was formulated in such a way so as to be as strong and general as possible, but at the cost of having a technical formulation that may obscure the fact that it implies the simpler Theorem \ref{thm:constrPrepSimple}.  The corollary of Theorem \ref{thm:constrPrep} given below directly implies Theorem \ref{thm:constrPrepSimple} and its analog for $p=\infty$ described in \eqref{eq:constrPrepSimpleBdd}, and it generalizes the interpolation theorem \cite[Theorem 2.4]{CluckersMiller2}.

The proof of the corollary makes use of the following observation: for the set $\F$ from Theorem \ref{thm:constrPrep}, if $f\in\F$ is subanalytic, then the restriction of $f$ to $A$ is $\varphi$-prepared (as opposed to being in the more general form allowed by \eqref{eq:TkSum} and \eqref{eq:TkForm}).  This observation follows from the way the proof of Theorem \ref{thm:constrPrep} uses Proposition \ref{prop:constrRect}, and from the way the proof of Proposition \ref{prop:constrRect} uses Proposition \ref{prop:subRect}.


\begin{corollary}\label{cor:constrPrep}
Suppose that $P\subset(0,\infty]$, that $D\subset\RR^{m+n}$ is subanalytic, and that $\Phi$ is a finite set of triples $(f,\mu,q)$ for which $f:D\to\RR$ is constructible, $\mu:D\to\RR$ is subanalytic, and $q > 0$. Define $E = \Pi_m(D)$ and $\F = \{f : (f,\mu,q)\in\Phi\}$.  Then to each $f\in\F$ we may associate a function $f^*\in\C(D)$ in such a way so that the following  statements hold.
\begin{enumerate}{\setlength{\itemsep}{3pt}
\item
There exists an open partition $\A$ of $D$ over $\RR^m$ such that for each $A\in\A$ there exists a rational monomial map $\varphi$ on $A$ over $\RR^m$ such that for every $(f,\mu,q)\in\Phi$, the function $\mu$ is $\varphi$-prepared and we may express $f^*$ as a finite sum
\begin{equation}\label{eq:Tk*}
f^*(x,y) = \sum_k T_k(x,y)
\end{equation}
on $A$, where each function $T_k$ is of the form \eqref{eq:TkForm}.

\item
The following hold for all $(f,\mu,q)\in\Phi$.
\begin{enumerate}{\setlength{\itemsep}{3pt}
\item
We have $f = f^*$ on $\{(x,y)\in D : P\subset\LC(f,|\mu|^q,x)\}$.

\item
For all $A\in\A$ and all terms $T_k$ in the sum \eqref{eq:Tk*}, we have $\Pi_m(A)\times P \subset \LC(T_k, |\mu|^q, \Pi_m(A))$.  (Hence $E\times P \subset \LC(f^*,|\mu|^q,E)$.)
}\end{enumerate}

\item
If $\infty\not\in P$, then we may take each function $T_k$ to be of the simpler form
\begin{equation}\label{eq:Tk*SimplerForm}
T_k(x,y) = g_k(x) \left(\prod_{i=1}^{n} |\tld{y}_{i}|^{r_{k,i}} \left(\log |\tld{y}_i|\right)^{s_{k,i}}\right) u_k(x,y),
\end{equation}
and the fact that $\Pi_m(A)\times P \subset \LC(T_k, |\mu|^q, \Pi_m(A))$ only depends on the values of the $r_{k,i}$, and not the values of the $s_{k,i}$, in the following sense: we have $\Pi_m(A)\times P \subset \LC(T'_k, |\mu|^q, \Pi_m(A))$ for any function $T'_k$ on $A$ of the form
\[
T'_k(x,y) = \prod_{i=1}^{n}|\tld{y}_{i}|^{r_{k,i}} \left(\log|\tld{y}_i|\right)^{s'_{k,i}},
\]
where the $r_{k,j}$ are as in \eqref{eq:Tk*SimplerForm} and the $s'_{k,i}$ are arbitrary natural numbers.
}\end{enumerate}
\end{corollary}

\begin{proof}
Let $\A$ be the open partition of $D$ obtained by applying Theorem \ref{thm:constrPrep} to $\Phi$; we use the notation of the theorem.  Because $\mu$ is subanalytic for every $(f,\mu,q)\in\Phi$, it follows that we may partition the members of $\A$ further in the $x$-variables to assume that for each $A\in\A$ and each $(f,\mu,q)\in\Phi$, either $\mu(x,y) = 0$ for all $(x,y)\in A$, or else for each $x\in\Pi_m(A)$ there exists $y\in A_x$ such that $\mu(x,y) \neq 0$.  Therefore for all $(f,\mu,q)\in\Phi$, $I\in\I(f,\mu,q)$, $A\in\A$ and $K\in\P(f,A)$, at least one of the following two statements holds.
\begin{enumerate}{\setlength{\itemsep}{3pt}
\item
For every $k\in K$ we have $\Pi_m(A)\times I \subset \LC(T_k, |\mu|^q, \Pi_m(A))$.

\item
We have $\sum_{k\in K} T_k(x,y) = 0$ on $\{(x,y)\in A : I\subset\LC(f,|\mu|^q,x)\}$.
}\end{enumerate}
For each $(f,\mu,q)\in\Phi$ and $A\in\A$, define $K^*(f,A)$ to be the union of all $K\in\P(f,A)$ for which there exists $I\in\I(f,\mu,q)$ such that $P\subset I$ and the above statement 1 holds.  For each $(f,\mu,q)\in\Phi$, define $f^*$ by
\[
f^*(x,y)
=
\begin{cases}
\sum_{k\in K^*(f,A)} T_k(x,y),
    & \text{if $(x,y)\in A$ with $A\in\A$,}
    \\
f(x,y),
    & \text{if $(x,y) \in D\setminus\bigcup\A$.}
\end{cases}
\]
Observe that statements 1 and 2 of the corollary hold.

To prove statement 3, suppose that $\infty\not\in P$.  By writing
\[
\log \prod_{j=1}^{n} |\tld{y}_j|^{\beta_{i,j}} = \sum_{j=1}^{n} \beta_{i,j} \log |\tld{y}_j|
\]
in \eqref{eq:TkForm} and then distributing, we may write each term $T_k$ as a finite sum of terms of the form \eqref{eq:Tk*SimplerForm} with the same values of the $r_{k,i}$ but possibly different values of the the $s_{k,i}$.  But only the values of the $r_{k,i}$ are relevant by \eqref{eq:T'k} since $\infty\not\in P$.
\end{proof}

%% file: countEx.tex
\section{A Counterexample for $L^\infty$ Spaces}\label{s:countEx}

The analog of Theorem \ref{thm:constrPrepSimple} for $p=\infty$ mentioned in the Introduction can be stated as follows: if $D\subset\RR^{m+n}$ is subanalytic and $f\in\C(D)$ is such that $\Int^\infty(f,\Pi_m(D)) = \Pi_m(D)$, then there exists an open partition $\A$ of $D$ over $\RR^m$ into cells over $\RR^m$ such that for every $A\in\A$ we may express $f$ as a finite sum $f(x,y) = \sum_k T_k(x,y)$ on $A$ for terms $T_k$ with $\LC^{\infty}(T_k,\Pi_m(A)) = \Pi_m(A)$ that are of the form
\begin{equation}\label{eq:TkGeneral}
T_k(x,y) = g_k(x)\left(\prod_{i=1}^{n} |\tld{y}_i|^{r_i} \left(\log \prod_{j=1}^{n} |\tld{y}_j|^{\beta_{i,j}} \right)^{s_{k,i}}\right) u_k(x,y),
\end{equation}
as denoted in the previous section.  This statement was proven in Corollary \ref{cor:constrPrep}.  A more literal analog of Theorem \ref{thm:constrPrepSimple} for $p = \infty$ would require the terms $T_k$ to be of the simpler form
\begin{equation}\label{eq:TkSimple}
T_k(x,y) = g_k(x)\left(\prod_{i=1}^{n} |\tld{y}_i|^{r_i} \left(\log|\tld{y}_i|\right)^{s_{k,i}}\right) u_k(x,y);
\end{equation}
however, this more literal analog is false, and the purpose of this section is to prove this by giving a counterexample.  It follows that in Statement 3 of Corollary \ref{cor:constrPrep}, one may not drop the assumption that $\infty\not\in P$; and in Theorem \ref{thm:constrPrep}, one may not replace \eqref{eq:T'k} with the statement $\Pi_m(A)\times I \subset \LC(T'_\kappa, |T'_\lambda|^q, \Pi_m(A))$.

For the rest of the section, write $(x,y) = (x,y_1,y_2)$ for coordinates on $\RR^3$, and define $f:D\to\RR$ by
\begin{equation}\label{eq:f}
f(x,y) = \log\left(\frac{y_1}{y_2}\right),
\end{equation}
where
\begin{equation}\label{eq:D}
D =\{(x,y) \in \RR^3 : 0 < x < 1, 0 < y_1 < 1, xy_1< y_2 < y_1\}.
\end{equation}
Note that the function $f(x,\cdot)$ is bounded on $D_x$ for every $x\in(0,1)$, and that the function $f$ is already a single term of the form given in \eqref{eq:TkGeneral} on $D$.   The obvious way to express $f$ as a sum of terms of the form \eqref{eq:TkSimple} is to write
\[
f(x,y) = \log y_1 - \log y_2
\]
on $D$; however, the terms $\log y_1$ and $\log y_2$ now become unbounded on each fiber $D_x$.  It should therefore seem feasible that $f$ is a counterexample for the more literal analog of Theorem \ref{thm:constrPrepSimple} for $p = \infty$.  To show that this is in fact the case, we prove the following assertion.

\begin{assertion}\label{ass:countEx}
For the function $f:D\to\RR$ defined in \eqref{eq:f} and \eqref{eq:D}, there does \underline{not} exist an open cover $\A$ of $D$ over $\RR$ such that for each $A\in\A$, $f$ may be written as a finite sum of terms $T_k$ of the form \eqref{eq:TkSimple} with each $T_k(x,\cdot)$ bounded on $A_x$ for all $x\in\Pi_m(A)$.
\end{assertion}

The proof of this assertion relies on the following lemma.

\begin{lemma}\label{lemma:countEx}
Let
\[
A = \{(x,z)\in\RR^2 : 0 < x < 1, x < z < 1\},
\]
and define an analytic isomorphism $\eta:(0,1)^2\to A$ by
\[
\eta(x,t) = (x, x^t).
\]
Suppose that $g:A\to\RR$ is a function of the form
\begin{equation}\label{eq:gCountEx}
g(x,z) = \sum_{i\in I} (\log x)^i x^{\alpha_i} z^{\beta_i} g_i(x,z)
\end{equation}
where $I\subset\NN$ is finite and nonempty, the $\alpha_i$ and $\beta_i$ are integers, and each $g_i$ is a function on $A$ that is not identically zero and is of the form
\[
g_i(x,z) = G_i\left(x,z,\frac{x}{z}\right)
\]
for an analytic function $G_i$ on $[0,1]^3$ represented by a single convergent power series, say
\[
G_i(X) = \sum_{\gamma\in\NN^3} G_{i,\gamma} X^\gamma,
\quad\text{for $X\in[0,1]^3$.}
\]
Then there exist $\epsilon\in(0,1]$, a nonzero real number $a$, a natural number $r$, and integers $p$ and $q$ such that for all $t\in(0,\epsilon)$,
\begin{equation}\label{eq:countExLimit}
\lim_{x\to 0} \frac{g\circ\eta(x,t)}{x^{p+qt}(\log x)^r} = a.
\end{equation}
\end{lemma}

\begin{proof}
By factoring out the lowest powers of $x$ and $z$ in \eqref{eq:gCountEx}, we may assume that the $\alpha_i$ and $\beta_i$ are all natural numbers.  But then each monomial $x^{\alpha_i} z^{\beta_i}$ can be incorporated into the function $g_i$, so we may in fact assume that the numbers $\alpha_i$ and $\beta_i$ are all zero.  For each $i\in I$,
\[
g_i\circ\eta(x,t)
=
G_i(x,x^t,x^{1-t})
=
\sum_{\gamma\in\NN^3} G_{i,\gamma} x^{\gamma_1 + t\gamma_2 + (1-t)\gamma_3}
=
\sum_{k=0}^{\infty}\sum_{l=-k}^{\infty} G_{i}^{[k,l]} x^{k+lt},
\]
where
\[
G_{i}^{[k,l]} = \sum_{\scriptstyle \gamma\in\NN^3\,\text{s.t.} \atop \scriptstyle \gamma_1 + \gamma_3 = k, \gamma_2 - \gamma_3 = l} G_{i,\gamma}.
\]
So
\begin{equation}\label{eq:gEta}
g\circ\eta(x,t)
=
\sum_{i\in I} (\log x)^i g_i\circ\eta(x,t)
=
\sum_{i\in I} \sum_{k=0}^{\infty}\sum_{l=-k}^{\infty} G_{i}^{[k,l]} x^{k+lt}(\log x)^i.
\end{equation}
Note that for each $i\in I$, the function $g_i$ is not identically zero and $\eta$ is a bijection, so $g_i\circ\eta$ is not identically zero, which implies that $G_{i}^{[k,l]}\neq 0$ for some $k$ and $l$.

Let $(p,q)$ be the lexicographically minimum member of the set
\begin{equation}\label{eq:kl}
\bigcup_{i\in I}\{(k,l)\in\NN\times\ZZ : \text{$k+l\geq 0$ and $G_{i}^{[k,l]}\neq 0$}\},
\end{equation}
and define $r = \max\{i\in I : G_{i}^{[p,q]}\neq 0\}$, $a = G_{r}^{[p,q]}$, and $\epsilon = \frac{1}{p+q+1}$.  We claim that for all $(k,l) \neq (p,q)$ in the set \eqref{eq:kl} and all $t\in(0,\epsilon)$,
\begin{equation}\label{eq:klIneq}
k+lt > p+ qt.
\end{equation}
The claim and \eqref{eq:gEta} together imply \eqref{eq:countExLimit}.  To prove the claim,  consider $(k,l) \neq (p,q)$ in \eqref{eq:kl}.  If $k=p$, then $l > q$, in which case \eqref{eq:klIneq} holds for all $t > 0$.   So suppose that $k\geq p+1$.  Simplifying the inequality $(p+1)(1-t) > p + qt$ shows that it is equivalent to the inequality $t < \epsilon$. So for all $t\in(0,\epsilon)$,
\[
k+lt = k(1-t) + (k+l)t \geq (p+1)(1-t) + 0t > p+qt,
\]
which proves the claim.
\end{proof}

In the following proof, we shall say that two functions $g,h:A\to\RR\setminus\{0\}$ are {\bf\emph{equivalent on $A$}} if the range of $g/h$ is contained in a compact subset of $(0,\infty)$.

\begin{proof}[Proof of Assertion \ref{ass:countEx}]
Suppose for a contradiction that there exists an open cover $\A'$ of $D$ over $\RR$ such that for each $A'\in\A'$, $f$ may be written as a finite sum $f(x,y) = \sum_k T_k(x,y)$ on $A'$ for terms $T_k$ of the form \eqref{eq:TkSimple} with each $T_k(x,\cdot)$ bounded on $A'_x$ for all $x\in\Pi_m(A')$; note that we associate to $A'$ a certain rational monomial map $\varphi'$ on $A'$ over $\RR$ that is used to defined the terms $T_k$.  By Proposition \ref{prop:subPrep} there exists an open cover $\A$ of $D$ over $\RR^0$ such that for each $A\in\A$ there exist a unique $A'\in\A'$ containing $A$ and a prepared rational monomial map $\varphi$ on $A$ over $\RR^0$ such that for each function $g_k$ occurring in \eqref{eq:TkSimple}, say of the form
\begin{equation}\label{eq:gkForm}
g_k(x) = \sum_i g_{k,i}(x) \prod_j \log g_{k,i,j}(x)
\end{equation}
for subanalytic functions $g_{k,i}$ and $g_{k,i,j}$,  the functions $g_{k,i}$ and $g_{k,i,j}$ are all $\varphi_{\leq 1}$-prepared on $\Pi_1(A)$.

The functions $xy_1$ and $y_1$ are not equivalent for $x$ near $0$, so we may fix $A\in\A$ of the form
\[
A = \{(x,y) : 0 < x < b_0, 0 < y_1 < b_1(x), a_2(x,y_1) < y_2 < b_2(x,y_1)\}
\]
with $a_2$ and $b_2$ not equivalent on $\Pi_2(A)$.  Let $\varphi$ be the rational monomial map on $A$ over $\RR^0$ associated with $A$.  Note that $x$ is not equivalent on $\Pi_1(A)$ to a constant, that $y_1$ is not equivalent on $\Pi_2(A)$ to a function of $x$, and that $y_2$ is not equivalent on $A$ to a function of $(x,y_1)$, so $\varphi$ must have center $0$.  For the same reason, if $A'$ is the unique member of $\A'$ containing $A$, and if $\varphi'$ is the rational monomial map over $\RR$ associated with $A'$, then $\varphi'$ must also have center $0$.  We are only interested in the restriction of $f$ to $A$, so we may therefore simply assume that $A' = A$ and $\varphi = \varphi'$.  So we may write
\begin{equation}\label{eq:log(y_1/y_2)}
\log\left(\frac{y_1}{y_2}\right)
=
\sum_k g_k(x) y_{1}^{r_{k,1}} y_{2}^{r_{k,2}} (\log y_1)^{s_{k,1}} (\log y_2)^{s_{k,2}} u_k(x,y_1,y_2)
\end{equation}
on $A$ for the constructible functions $g_k$ given in \eqref{eq:gkForm}, rational numbers $r_{k,1}$ and $r_{k,2}$, natural numbers $s_{k,1}$ and $s_{k,2}$, and $\varphi$-units $u_k$; and we may write
\[
a_2(x,y_1) = x^\alpha y_1 u(x,y_1)
\quad\text{and}\quad
b_2(x,y_1) = x^\beta y_1 v(x,y_1)
\]
on $\Pi_2(A)$ for some rational numbers $\alpha$ and $\beta$ satisfying $0 \leq \beta < \alpha\leq 1$ and some $\varphi_{\leq 2}$-units $u$ and $v$.

Fix positive constants $c$ and $d$ satisfying $c > u(x,y_1)$ and $d < v(x,y_1)$ on $\Pi_2(A)$.  Since $\alpha > \beta$, by shrinking $b_0$ we may assume that
\[
A = \{(x,y) : 0 < x < b_0, 0 < y_1 < b_1(x), c x^\alpha y_1 < y_2 < d x^\beta y_1\}.
\]
Pulling back the equation \eqref{eq:log(y_1/y_2)} by the map $(x,y_1,y_2)\mapsto (x,y_1,y_1y_2)$ gives
\begin{equation}\label{eq:log(1/y_2)}
\log\left(\frac{1}{y_2}\right)
=
\sum_k g_k(x) y_{1}^{r_{k,1} + r_{k,2}} y_{2}^{r_{k,2}} (\log y_1)^{s_{k,1}+s_{k,2}}
\left(1 + \frac{\log y_2}{\log y_1}\right)^{s_{k,2}} u_k(x,y_1,y_1y_2)
\end{equation}
on the set
\[
\{(x,y_1,y_2) : 0 < x < b_0, 0 < y_1 < b_1(x), c x^\alpha < y_2 < d x^\beta\}.
\]
By assumption, each term of \eqref{eq:log(1/y_2)} is bounded for each fixed value of $x$, so letting $y_1$ tend to $0$ for each fixed value of $(x,y_2)$ shows that for each $k$, either $r_{k,1} + r_{k,2} > 0$ or $r_{k,1} + r_{k,2} = s_{k,1} + s_{k,2} = 0$ (and $s_{k,1} + s_{k,2} = 0$ means that $s_{k,1} = s_{k,2} = 0$).  So letting $y_1$ tend to $0$ in \eqref{eq:log(1/y_2)}
gives
\begin{equation}\label{eq:log(1/y_2)Limit}
\log\left(\frac{1}{y_2}\right) = \sum_k g_k(x) y_{2}^{r_{k,2}} v_k(x,y_2)
\end{equation}
on
\[
\{(x,y_2) : 0 < x < b_0, cx^\alpha < y_2 < dx^\beta\},
\]
where each $v_k$ is a $\psi$-unit with $\psi$ defined by $\psi(x,y_2) = \lim_{y_1\to 0}\varphi(x,y_1,y_1y_2)$.

By pulling back \eqref{eq:log(1/y_2)Limit} by the map $(x,y_2)\mapsto (x, cx^\beta y_{2}^{\alpha-\beta})$ and expanding logarithms using \eqref{eq:gkForm}, we may write
\begin{equation}
\log y_2 = \sum_i (\log x)^i x^{\alpha_i} y_{2}^{\beta_i} f_i(x,y_2)
\end{equation}
on
\begin{equation}\label{eq:setC}
\{(x,y_2) : 0 < x < b_0, x < y_2 < C\}
\end{equation}
for some $C > 0$, rational numbers $\alpha_i$ and $\beta_i$, and $\psi$-functions $f_i$ (for an appropriately modified $\psi$), where $i$ ranges over some finite set of natural numbers.  By pulling back by $(x,y_2)\mapsto (x^r, y_{2}^{r})$ for a suitable positive integer $r$, we may further assume that all the $\alpha_i$ and $\beta_i$ are integers, and that the components of $\psi(x,y_2)$ are also all monomial in $(x,y_2)$ with integer powers.  Thus each component of $\psi$ is either of the form $x^p$ for some positive integer $p$, is of the form $y_{2}^{q}$ for some positive integer $q$, or is of the form $x^p/y_{2}^{q} = x^{p-q}(x/y_2)^q$ for some positive integers $p$ and $q$ with $p\geq q$.  So we may assume that $\psi(x,y_2) = (x,y_2,x/y_2)$, and therefore write $f_i(x,y_2) = F_i(x,y_2,x/y_2)$ for some analytic function $F_i$ defined on the closure of $\{(x,y_2,x/y_2) : (x,y_2)\in A\}$.  Fix $\delta > 0$ sufficiently small so that
\begin{equation}\label{eq:delta}
\{(x,y_2) : 0 < x < \delta^2, 0 < y_2 < \delta, x/y_2 < \delta\}
\end{equation}
is contained in \eqref{eq:setC} and that $F_i$ is represented by a single convergent power series on $[-\delta^2,\delta^2]\times[-\delta,\delta]\times[-\delta,\delta]$.  Thus restricting to \eqref{eq:delta} and then pulling back by $(x,y_2)\mapsto (\delta^2 x, \delta y_2)$ gives
an equation of the form
\begin{equation}\label{eq:log(y_2)}
\log y_2 = \sum_i (\log x)^i x^{\alpha_i} y_{2}^{\beta_i} F_i\left(x,y_2,\frac{x}{y_2}\right)
\end{equation}
on
\[
\{(x,y_2) : 0 < x < 1, x < y_2 < 1\},
\]
with each $F_i$ represented by a single convergent power series on $[-1,1]^3$ centered at the origin.

Applying Lemma \ref{lemma:countEx} to the right side of \eqref{eq:log(y_2)} shows that there exist $\epsilon\in(0,1]$, a nonzero real number $a$, a natural number $r$, and integers $p$ and $q$ such that for all $t\in(0,\epsilon)$,
\[
\lim_{x\to 0} \frac{t\log x}{x^{p+qt}(\log x)^r} = a.
\]
Considering this limit for any fixed value of $t\in(0,\epsilon)$ shows that $r = 1$ and that $p+qt = 0$, so in fact $p = q = 0$ since $t\in(0,\epsilon)$ is arbitrary.  But then $t = a$ for all $t\in(0,\epsilon)$, which is a contradiction that completes the proof.
\end{proof}

%% file: acknowledgment.tex
\section*{Acknowledgement}

During the preparation of this paper, the research of the first author
has been partially supported by the Fund for Scientific Research - Flanders (G.0415.10) - Belgium, and the research of the second author has been partially supported by the National Science Foundation (award number 1101248) - U.S.A. 